\documentclass[10pt]{amsart}
\usepackage{amscd}  
\usepackage{amsmath}
\usepackage{amsthm}
\usepackage{amssymb}

\newcommand{\norm}[1]{\Vert #1 \Vert}  
\newcommand{\abs}[1]{\vert #1 \vert}  
\newcommand{\ov}{\overline}
\newcommand{\gr}{\mathrm{grad}}
\newcommand{\Ker}{\mathrm{Ker}\,}
\newcommand{\dif}{\mathrm{d}}
\newcommand{\Dd}{\mathfrak{D}}

\newcommand{\di}{\mathrm{div}}
\newcommand{\tr}{\mathrm{trace}}
\newcommand{\ii}{\mathrm{i}}
\newcommand{\CC}{\mathbb{C}}
\newcommand{\Cc}{\mathfrak{C}}
\newcommand{\TT}{\mathbb{T}}
\newcommand{\HH}{\mathcal{H}}
\newcommand{\VV}{\mathcal{V}}
 
\newcommand{\DD}{\mathcal{D}} 
\newcommand{\RR}{\mathbb{R}}   
\newcommand{\ZZ}{\mathbb{Z}} 

\newcommand{\Ss}{\mathbb{S}} 

\numberwithin{equation}{section}

\newtheorem{te}{Theorem}[section]
\newtheorem{pr}{Proposition}[section]
\newtheorem{co}{Corollary}[section]

\theoremstyle{definition}
\newtheorem{de}{Definition}[section]    
\newtheorem{re}{Remark}[section]
\newtheorem{ex}{Example}[section]

\begin{document}

\title{On the geometrized Skyrme and Faddeev models}

\author{Radu Slobodeanu}

\address{Department of Theoretical Physics and Mathematics, University of Bucharest, P.O. Box Mg 11, RO - 077125 Bucharest, Romania.}

\email{radualexandru.slobodeanu@g.unibuc.ro}

\date{\today}

\thanks{This research was supported by the CEx Grant no. 2-CEx 06-11-22/ 25.07.2006 and, during the preparation of the present revised version, by PN II Idei Grant, CNCSIS, code 1193.}


\begin{abstract}
The higher-power derivative terms involved in both Faddeev and Skyrme energy functionals correspond to $\sigma _2$--energy, introduced by Eells and Sampson in \cite{els}. The paper provides a detailed study of the first and second variation formulae associated to this energy. Some classes of (stable) critical points are outlined.
\end{abstract}

\maketitle

\section{Introduction}
Common tools in field theory, non-linear $\sigma$-models are known in differential geometry mainly through the problem of \textit{harmonic maps} between Riemannian manifolds. Namely a (smooth) mapping $\varphi:(M,g) \to (N,h)$ is harmonic if it is critical point for the Dirichlet energy functional \cite{els}, 
\begin{equation*}
\mathcal{E}(\varphi) =\frac{1}{2} \int_M \abs{\dif \varphi}^2 \nu_g,
\end{equation*}
a generalization of the kinetic energy of classical mechanics.

Less discussed from differential geometric point of view are Skyrme and Faddeev-Hopf models, which are $\sigma$--models with additional fourth-power derivative terms (for an overview including recent progress concerning both models, see \cite{lin}). 

The first one was proposed in the sixties by Tony Skyrme \cite{sky}, to model baryons as \textit{topological solitons} (see \cite{mant}) of pion fields. Meanwhile it has been shown \cite{wit} to be a low energy effective theory of quantum chromodynamics that becomes exact as the number of quark colours becomes large. Thus baryons are represented by energy minimising, topologically nontrivial maps $\varphi: \mathbb{R}^3 \to SU(2) \cong \mathbb{S}^3$ with the boundary condition $\varphi(\{|x| \to \infty \}) = I_2$, called \textit{skyrmions}. Their topological degree is identified with the \textit{baryon number}.
The static (conveniently renormalized) Skyrme energy functional is  
\begin{equation}\label{eskyrme}
\mathcal{E}_{\texttt{Skyrme}}(\varphi)=
\frac{1}{2} \int_{\RR^3} \left(\abs{\dif \varphi}^2 +
\frac{1}{2}\abs{\dif \varphi \wedge \dif \varphi}^2 \right)\dif^3 x.
\end{equation}
This energy has a topological lower bound \cite{fa}: $\mathcal{E}_{\texttt{Skyrme}}(\varphi) \geq 
6 \pi^2 \abs{\mathrm{deg}\varphi}$.

In the second one, stated in 1975 by Ludvig Faddeev and Antti J. Niemi \cite{fad}, the configuration fields are unitary vector fields $\varphi: \mathbb{R}^3 \to \mathbb{S}^2 \subset \mathbb{R}^3$ with the boundary condition $\varphi(\{|x| \to \infty \}) = (0,0,1)$. The static energy in this case is given by
\begin{equation}\label{efaddeev}
\mathcal{E}_{\texttt{Faddeev}} (\varphi)=\int_{\RR^3} \left( c_2 \abs{\dif \varphi}^2
+ c_4 \langle \dif \varphi \wedge \dif \varphi, \varphi \rangle ^2 \right) \dif^3 x,
\end{equation}
where $c_2, c_4$ are coupling constants.

Again the field configurations are indexed by an integer, their \textit{Hopf invariant}: $Q(\varphi) \in \pi_3(\mathbb{S}^2) \cong \mathbb{Z}$ and the energy has a topological lower bound:
$\mathcal{E}_{\texttt{Faddeev}}(\varphi) \geq c \cdot \abs{Q(\varphi)}^{3/4}$, cf. \cite{vaku}. Although this model can be viewed as a constrained variant of the Skyrme model, it exhibits important specific properties, e.g. it allows \textit{knotted solitons}.
Moreover, in \cite{dual} it has been proposed that it arises as a dual description of strongly coupled $SU(2)$ Yang-Mills theory, with the solitonic strings (possibly) representing glueballs. See also \cite{dual+} for an alternative approach to these issues.

\medskip
Both models rise the same kind of topologically constrained minimization problem: find out static energy minimizers in each topological class (i.e. of prescribed baryon number or Hopf invariant). We can give an unitary treatment for both if we take into account that they are particular cases of the following energy-type functional:

\begin{equation}\label{generg}
\mathcal{E}_{\sigma_{1,2}}:\mathcal{C}^{\infty}(M,N) \to \RR_{+}, \qquad \mathcal{E}_{\sigma_{1,2}}(\varphi) =
\frac{1}{2} \int_M \left[ \abs{\dif \varphi}^2 + 
\kappa \cdot \sigma_2(\varphi)\right]  \nu _g,
\end{equation}
where $(M,g)$, $(N,h)$ are (smooth) Riemannian manifolds, $\kappa \geq 0$ is a coupling constant and $\sigma_2(\varphi)$ is the second elementary symmetric function of the eigenvalues of $\varphi^* h$ with respect to $g$. 

Even if the variational problem for the $\sigma_p$-energy has already been treated in \cite{sacs, ee, cri}, very little is known about its solutions. From our point of view, the particularities of $p=2$ case are worth to be outlined for their differential geometric interest in its own and hopefully for providing hints for further investigations on the original physical models.

The present generalization of \eqref{eskyrme} and \eqref{efaddeev} was proposed in \cite{los, man}. Other generalizations of Skyrme and Faddeev energies are discussed in \cite{au, auk, fer, kos, sve, yang}.

The paper is organized as follows. The next section reviews the higher power energies in terms of eigenvalues of the Cauchy-Green tensor and some classes of mappings characterized by their distortion. In section 3 the Euler-Lagrange equations for $\sigma_2$--energy are derived and  general solutions inside these classes are pointed out. The effect of (bi)conformal changes of domain metric is also stated. Section 4 presents the second variation formula and analyses the stability of homothetic and holomorphic solutions. Finally we apply the results of previous sections to the stability of homotheties for the full energy \eqref{generg} and to some old and new \textit{ansatze} for stationary field configurations. We end with an example of absolute minima for the strongly coupled Faddeev model on $\Ss^3$ and a discussion on possible contactomorphic solutions.

\section{Higher power energies and the Cauchy-Green tensor}

\subsection{The Cauchy-Green tensor and the geometrical distortion induced by a map}
Let $\varphi: (M^{m}, g) \to (N^{n}, h)$ a smooth mapping between Riemannian manifolds of dimensions $m$ and $n$. 
The so called {\it first fundamental form} of $\varphi$ is the symmetric, positive semidefinite 2-covariant tensor field on $M$, 
defined as $\varphi^{*}h$, cf. \cite{ee}.
Alternatively, using the musical isomorphism, we can see it as the endomorphism $\Cc _{\varphi} = \dif \varphi^{t} \circ \dif \varphi: TM \rightarrow TM$, where $\dif \varphi^{t}:TN \rightarrow TM$ denote {\it the adjoint} of $\dif \varphi$. When $m=n=3$, this corresponds to the \textit{(right) Cauchy-Green (strain) tensor} of a deformation in non-linear elasticity (we shall maintain this name for $\Cc _{\varphi}$ in the general case). 

The Cauchy-Green tensor is always diagonalizable; 
let $\lambda_{1}^{2}$, $\lambda_{2}^{2}$, ..., $\lambda_{r}^{2}$ and $\lambda_{r+1}^{2}=...=\lambda_{m}^{2}=0$
be its (real, non-negative) eigenvalues, where $r := \mathrm{rank} (\dif \varphi)$ everywhere. Recall that $\lambda_{i}=\sqrt{\lambda_{i}^{2}}$ are also called \textit{principal distortion coefficients of} $\varphi$.

The elementary symmetric functions in the eigenvalues of $\varphi^{*}h$ represent a measure of the geometrical distortion induced by the map.
They are called \textit{principal invariants}
of $\dif \varphi$ and will be denoted by:
$$
\sigma_1(\varphi) = \sum_{i=1}^{m}\lambda_{i}^{2};
\quad \sigma_2(\varphi) = \sum_{i<j=1}^{m}\lambda_{i}^{2}\lambda_{j}^{2};
\quad . . . \ ; \quad \sigma_m(\varphi) = \lambda_{1}^{2}\lambda_{2}^{2} \cdots \lambda_{m}^{2},
$$
or, alternatively:
$$
\sigma_1(\varphi)= 2e(\varphi); \qquad
\sigma_2(\varphi) = \abs{\wedge^2 \dif \varphi}^2; \quad . . . \ ;
\quad \sigma_m(\varphi)= [v(\varphi)]^2,
$$
where $e(\varphi)=\frac{1}{2}\abs{\dif \varphi}^2$ is the \textit{energy density of} $\varphi$ and $v(\varphi)=\sqrt{\mathrm{det}(\varphi^* h)}$ is the \textit{volume density of} $\varphi$, cf. \cite{els}.

\begin{re}
At any point of $M$, there is an orthonormal basis $\{ e_{i} \}$ of corresponding eigenvectors for $\varphi^{*}h$ at that point. Moreover, according to \cite[Lemma 2.3]{pant}, we have a \textit{local} orthonormal frame of eigenvector fields, around any point of a dense open subset of $M$. In particular, for such local "eigenfields" we have:
$\varphi^{*}h(e_{i}, e_{j})=\delta_{ij}\lambda_{i}^{2}$, so $\{ \dif \varphi(e_{i}) \}$ are orthogonal with norm $\norm{\dif \varphi(e_{i})}=\lambda_{i}$. 
\end{re}

\subsection{Higher power energies}
According to \cite{els}, up to a half factor, we shall call \textit{$\sigma_p$--energy}, the following functional
\begin{equation} \label{sp}
\mathcal{E}_{\sigma_p}(\varphi) = \frac{1}{2} 
\int_M \sigma_p(\varphi) \nu _g.
\end{equation}

Therefore, the generalized energy \eqref{generg} reads
\begin{equation} \label{sky}
\mathcal{E}_{\sigma_{1,2}}(\varphi) = 
\mathcal{E}_{\sigma_1}(\varphi)+\kappa \mathcal{E}_{\sigma_2}(\varphi) = \frac{1}{2} 
\int_M \left( \sum_{i}\lambda_{i}^{2} + \kappa \sum_{i < j}\lambda_{i}^{2}\lambda_{j}^{2} \right) \nu _g.
\end{equation}

\bigskip

Let us recall another type of (higher power) energy-type functional that will be useful for our further discussion. The \textit{$p$-energy} of a (smooth) map is defined as:
\begin{equation*}
\mathcal{E}_p(\varphi)=\frac{1}{p}\int_{M}\abs{\dif \varphi}^p \nu _g
\end{equation*}

The corresponding Euler-Lagrange operator/equations are, cf. \cite{take}
\begin{equation*}
\tau_p(\varphi):=\abs{\dif \varphi}^{p-2}\left[\tau(\varphi)+
(p-2)\dif \varphi (\mathrm{grad}(\ln \abs{\dif \varphi}))\right]
\equiv 0,
\end{equation*}
where $\tau(\varphi):=\mathrm{trace}\nabla \dif \varphi$ is the \textit{tension field} of $\varphi$ (i.e. the Euler-Lagrange operator associated to the Dirichlet energy). The solutions of these equations are called \textit{p-harmonic} maps.

In particular, for $p=4$, we have  
\begin{equation}\label{21}
\abs{\dif \varphi}^{2} \left[\tau(\varphi)+
2\dif \varphi (\mathrm{grad}(\ln \abs{\dif \varphi}))\right]=0,
\end{equation}
or, equivalently, $e(\varphi)\tau(\varphi)+
\dif \varphi (\mathrm{grad}(e(\varphi))=0$.

\begin{re}
It is easy to see that $\mathcal{E}_{\sigma_{2}}(\varphi) = 
\frac{1}{4} \int_M (\abs{\dif \varphi}^4 - \abs{\varphi^* h}^2) \nu _g = \mathcal{E}_{4}(\varphi) - \frac{1}{4} \int_M \abs{\varphi^* h}^2 \nu _g$. The relation with the $4$-energy is clearer if we point out that, using \textit{Newton's inequalities},
\begin{equation*}
\mathcal{E}_{\sigma_2}(\varphi) \leq \frac{n-1}{n}\mathcal{E}_4(\varphi)
\end{equation*}
with equality if and only if $\lambda_1= ... =\lambda_n$. If in addition $\varphi$ is of bounded dilation, i.e. $\lambda_{1}^{2}/\lambda_{2}^{2} \leq K^{2}$, we have also the reversed inequality 
\begin{equation*}
\frac{2}{n^2 K^2}\mathcal{E}_4(\varphi) \leq \mathcal{E}_{\sigma_2}(\varphi).
\end{equation*}
\end{re}

\subsection{Classes of mappings characterized by their distortion} Let us recall some classes of mappings that will have a particular behaviour with respect to the above mentioned energies. For the  contact or symplectic geometry background and corresponding notations we refer the reader to \cite{ble}.
\begin{enumerate}
\item (Equal eigenvalues.)
$(i)$ When $m \geq n$ and $r \in\{0,n\}$,  if $\lambda_{1}^{2}=...=\lambda_{r}^{2}=\lambda^{2}$, we say that our map is \textit{horizontally weakly conformal} (HWC) or \textit{semiconformal} of \textit{dilation} $\lambda$, cf. \cite[p. 46]{ud} (when the map is submersive, we shall omit the word "weakly"). If moreover $\gr \lambda \in \Ker \dif \varphi$ (at regular points), then the map is called \textit{horizontally homothetic} (HH); in the particular case when $\gr \lambda = 0$, we call it simply \textit{homothetic}. 

$(ii)$ When $m \leq n$ and $r \in\{0, m\}$,  if $\lambda_{1}^{2}=...=\lambda_{r}^{2}=\lambda^{2}$, we say that our map is \textit{(weakly) conformal}, cf. \cite[p. 40]{ud}. If $m=n$ this notion is equivalent to the above one.

\item (Pairwise equal eigenvalues.) When $N$ is endowed with an almost Hermitian structure $J$ (so $n$ is even), a class of mappings that includes the above ones was defined by $[\dif \varphi \circ \dif \varphi ^t, J]=0$, cf. \cite{lub}. These maps are called \textit{pseudo horizontally weakly conformal maps} (PHWC).  In this case, cf. \cite{slub}, the eigenvalues of $\varphi^* h$ have multiplicity 2, i.e. $\lambda_{1}^{2}=\lambda_{2}^{2}$, $\lambda_{3}^{2}=\lambda_{4}^{2}$, ... $\lambda_{r-1}^{2}=\lambda_{r}^{2}$, $r$ is even, the eigenspaces are invariant with respect to the induced metric almost $f$--structure, $F^{\varphi}$, on the domain and $\varphi$ is $(F^{\varphi}, J)$-holomorphic, i.e. $\dif \varphi \circ F^{\varphi} = J \circ \dif \varphi$. When $F^{\varphi}$ is an almost complex structure we recover the classical  case of holomorphic maps. \\
We have also a corresponding notion of \textit{pseudo horizontally homothetic} (PHH) map \cite{aa, aab}. For submersions, PHH condition reads as $(\nabla^{\HH}_{X}F^{\varphi})(Y) = 0$, $\forall X, Y \in \Ker((F^{\varphi})^2 + I)$. Standard examples of PHH maps are the holomorphic maps between K\"ahler manifolds and the $(\phi,J)$- holomorphic maps from a Sasakian manifold (with associated $f$-structure $\phi$) to a K\"ahler one. \\
For a suitable notion of PHWC/holomorphic map between odd-dimensional manifolds see \cite{slo} and references therein. We mention here only the fact that, according to \cite{ian}, a $(\phi, \phi^{\prime})$-holomorphic map between two contact metric manifolds, $(M^{2m+1}, \phi, \xi, \eta, g)$ and $(N^{2n+1}, \phi^{\prime}, \xi^{\prime}, \eta^{\prime}, h)$, has eigenvalues $\lambda_{1}^{2}=a^2$, $\lambda_{2}^{2}=\lambda_{3}^{2}=$ ... $=\lambda_{r}^{2}=a$ ($a >0$ and $r$ is odd). This kind of map is called \textit{contact homothety} or $\DD$-\textit{homothetic transformation}.

\item (Equal products of paired eigenvalues) When both the domain and codomain of $\varphi$ are symplectic manifolds, $(M^{2m}, \Omega, g)$ and $(N^{2n}, \Omega^{\prime}, h)$, and $\varphi^*\Omega^{\prime}=\Omega$ we say that $\varphi$ is a \textit{symplectomorphism}. According to \cite{medo}, Cauchy-Green's eigenvalues for a symplectomorphism satisfy 
$\lambda_{1}^{2} \cdot \lambda_{2}^{2} = \lambda_{3}^{2}\cdot \lambda_{4}^{2}=$ ... $=\lambda_{r-1}^{2}\cdot\lambda_{r}^{2}=1$ ($r$ is even) and the associated complex structure restricts to an isomorphism between the eigenspaces corresponding to $\lambda_{i}^{2}$ and $\lambda_{i+1}^{2}$, for all odd $i$. It is an easy task to rephrase this fact for 
$\varphi^*\Omega^{\prime}=a\cdot \Omega$ ($a$ being a function on $M$) and to see that the common value of products of eigenvalues must be equal to $a^2$. Notice that the conformal case (1) is included in this case.\\
Moreover one can extend this result for \textit{contactomorphisms}, that is for mappings between contact metric manifolds, $(M^{2m+1}, \eta, g)$ and $(N^{2n+1}, \eta^{\prime}, h)$, satisfying $\varphi^*\eta^{\prime}=a\cdot \eta$. In this case we obtain 
$$\lambda_{1}^{2} = \lambda_{2}^{2} \cdot \lambda_{3}^{2} = ... =\lambda_{r-1}^{2}\cdot\lambda_{r}^{2}=a^2,$$ 
where $r$ is odd and the eigenvector corresponding to $\lambda_{1}^{2}$ must be $\xi$, the Reeb vector field on the domain. Notice that the conformal case (1) is no more included in this case, but only contact homotheties.
\end{enumerate}

\section{Euler-Lagrange equations for $\sigma_2$--energy}

As in the previous section, $\varphi: (M,g) \to (N,h)$ will denote a smooth mapping between Riemannian manifolds. For the rest of the paper we suppose $M$ to be compact (unless otherwise stated) and we denote by $\nu _g$ the volume form of its metric. The distributions $\VV = \Ker \dif \varphi$ and $\HH=\VV^{\perp}$ on $M$ will be called \textit{vertical} and \textit{horizontal} spaces and the projections of a vector field $X$ along these distributions will be denoted as $X^{\VV}$ and $X^{\HH}$, respectively.

\subsection{The first variation formula} Let $\{\varphi_t\}$ be a (smooth) variation of $\varphi$ with variation vector field 
$v \in \Gamma(\varphi^{-1}TN)$, i.e.
$$
v(x) = \displaystyle{\frac{\partial \varphi_t}{\partial t}(x) \Bigl\lvert _{t=0}} \in T_{\varphi(x)}N, \qquad \forall x \in M.
$$ 

In this section, we are looking for critical points of $\sigma_2$--energy, i.e. mappings that satisfy $\dfrac{\dif}{\dif t} \Bigl\lvert _{0} \mathcal{E}_{\sigma_2}(\varphi_t)=0$, for any variation. For simplicity let us call these maps $\sigma_2$--\textit{critical}. Analogously, critical maps for the full functional \eqref{generg} will be called $\sigma_{1,2}$--\textit{critical}. 

\begin{re}\label{31}
($a$) \ To every $v \in \Gamma(\varphi^{-1}TN)$ we associate a vector field on $M$, $X_v \in (\Ker \dif \varphi)^{\perp}$, defined by:
\begin{equation*}
g(X_v, Y)=h(v, \dif \varphi(Y)), \quad \forall Y \in \Gamma(TM).
\end{equation*}
If $\varphi$ is a horizontally conformal (surjective) submersion of dilation $\lambda$, then $v=\lambda^{-2}\dif \varphi(X_v)$.

\medskip
\noindent ($b$) \ Denote $\alpha _v := h\left(\nabla^{\varphi}v, \dif \varphi\right)$ and $\mathrm{div}^{\varphi}v :=\tr \, \alpha _v$, where $\nabla^{\varphi}$ is the pull-back connection in $\varphi^{-1}TN$ (see \cite{ud}). Then it is easy to check that:

\medskip
$(i.)$ \  $\dfrac{\partial}{\partial t}\Bigl\lvert _{t=0} \varphi_{t}^{*}h(Y,Z) 
= \alpha _v (Y, Z) + \alpha _v (Z, Y)$;

\medskip
$(ii.)$ \  $\alpha _v (Y, Z) = g\left(\nabla_{Y}X_v, Z \right) -h(v, \nabla \dif \varphi(Y, Z))$;

\medskip
$(iii.)$ \  $\mathrm{div}^{\varphi}v = 
\mathrm{div}X_v - h(v, \tau(\varphi))$;

\medskip
$(iv.)$ \  $\varphi$ is harmonic if and only if $\mathrm{div}^{\varphi}v = \mathrm{div}X_v, \ \forall v \in \Gamma(\varphi^{-1}TN)$.
\end{re}

In a (local) orthonormal eigenvector frame $\{e_i \}_{i=1,...,m}$ for $\varphi^* h$, using the above remark, we can compute the first derivative of $\mathcal{E}_{\sigma_2}(\varphi_t)$ as follows:
\begin{align*}
\dfrac{\dif}{\dif t} \Bigl\lvert _{0}
\mathcal{E}_{\sigma_2}(\varphi_t)
=& \frac{1}{2}\int_M \sum_{i<j}\dfrac{\dif}{\dif t} \Bigl\lvert _{0} \left[ \norm{\dif \varphi _t(e_i)}^2\norm{\dif \varphi _t(e_j)}^2 
- h(\dif \varphi _t(e_i), \dif \varphi _t(e_j))^2 \right]\nu_ g\\
=& \frac{1}{2}\int_M \sum_{i}\lambda_{i}^{2}\dfrac{\dif}{\dif t} \Bigl\lvert _{0}\left(\abs{\dif \varphi _t}^2-\norm{\dif \varphi _t(e_i)}^2 \right)\nu _g \\
=& \int_M \sum_{i} \lambda_{i}^{2} \left\{ \di ^{\varphi}v -
\alpha _v (e_i, e_i) \right\} \nu _g\\
=& \int_M \left\{2e(\varphi) \di ^{\varphi}v -
\sum_{i} \lambda_{i}^{2} \left[ g(\nabla_{e_i}X_v, e_i)-h(v, \nabla \dif \varphi(e_i, e_i)) \right]\right\} \nu _g\\
=& \int_M  h(v,  -2[e(\varphi)\tau(\varphi) + 
\dif \varphi(\mathrm{grad}e(\varphi))]) \nu _g \\ 
&+\int_M \left\{h\left(v, \sum_i \lambda_{i}^{2} \nabla \dif \varphi (e_i , e_i)\right) - \sum_i \lambda_{i}^{2}g(\nabla_{e_i}X_v , e_i) \right\} \nu _g. 
\end{align*}

Denote $\widetilde X_v = \sum_i \lambda_{i}^{2}g(X_v , e_i)e_i$.
Then:
\begin{equation}\label{tric}
\begin{split}
\mathrm{div}\widetilde X_v - \sum_i \lambda_{i}^{2}g(\nabla_{e_i}X_v , e_i) =&g\left(X_v, \sum_k \left[e_k(\lambda_{k}^{2}) + \sum_i (\lambda_{i}^{2} - \lambda_{k}^{2}) g(\nabla_{e_i}e_i,e_k) \right]e_k\right)\\
=&h(v, \dif \varphi([\mathrm{div}\varphi^*h]^{\sharp})).
\end{split}
\end{equation}

\begin{re}
Let us rewrite two of the terms that appeared above as:

\noindent ($a$) $[\mathrm{div}\varphi^*h]^{\sharp}=\di \Cc_\varphi$;

\medskip
\noindent ($b$) $\sum_i \lambda_{i}^{2} \nabla \dif \varphi (e_i , e_i) = \tr (\nabla \dif \varphi) \circ \Cc_\varphi$, where the right hand term is defined in an \textit{arbitrary} orthonormal frame as follows:
\begin{equation*}
\tr (\nabla \dif \varphi) \circ \Cc_\varphi=
\sum_i \nabla \dif \varphi (e_i , \Cc_\varphi e_i)=
\sum_i \varphi^* h(e_i, e_i) \nabla \dif \varphi (e_i , e_i)
+2 \sum_{i<j} \varphi^* h(e_i, e_j) \nabla \dif \varphi (e_i , e_j).
\end{equation*} 
\end{re}

\begin{de}
We call \textit{$\sigma_2$--tension field} of the map $\varphi$ the following section of the pull-back bundle $\varphi^{-1}TN$:
$$\tau_{\sigma_2}(\varphi)=2[e(\varphi)\tau(\varphi) + 
\dif \varphi(\mathrm{grad}e(\varphi))]-
\tr (\nabla \dif \varphi) \circ \Cc_\varphi 
-\dif \varphi(\di \Cc_\varphi).$$
\end{de}
We have obtained the following (cf. also \cite{cri})

\begin{pr}[The first variation formula]
\begin{equation*}
\dfrac{\dif}{\dif t} \Bigl\lvert _{0}
\mathcal{E}_{\sigma_2}(\varphi_t)
=- \int_M h(v, \tau_{\sigma_2}(\varphi))\nu _g.
\end{equation*}

In particular, a map $\varphi$ is $\sigma_2$--critical if it satisfies the following Euler-Lagrange equations
\begin{equation}\label{EL}
2[e(\varphi)\tau(\varphi) + 
\dif \varphi(\mathrm{grad}e(\varphi))]-
\tr (\nabla \dif \varphi) \circ \Cc_\varphi  
-\dif \varphi(\di \Cc_\varphi)=0.
\end{equation}
\end{pr}

\begin{re}
The Euler-Lagrange operator of $\mathcal{E}_{\sigma_p}$ has been derived in \cite{cri} for all $p$:
\begin{equation*}
\tau_{\sigma_p}(\varphi)=\mathrm{trace}\nabla(\dif \varphi \circ \chi_{p-1}(\varphi)),
\end{equation*}

\noindent where $\chi_{p-1}(\varphi)$ is the \textit{Newton tensor}. In $p=2$ case, $\chi_{1}(\varphi)=2e(\varphi)Id_{TM}-\dif \varphi^{t} \circ \dif \varphi$ and then we can easily obtain the equation \eqref{EL}. Nevertheless, in this particular case, we preferred to derive the first variation \textit{ab initio}, for the sake of completeness (as it might be difficult to access \cite{cri}). Recall also that $\tau_{\sigma_2}$ is elliptic on $\{\varphi \in \mathcal{C}^2(M,N) \vert \mathrm{rank} \varphi > 2\}$, cf. \cite{cri}.

Analogously to the harmonic map problem, Euler-Lagrange equations can be written (at least for submersions) in the conservative form $\di S_{\sigma_2}(\varphi)=0$, where
\begin{equation*}
S_{\sigma_2}(\varphi)=\frac{1}{2}\sigma_2(\varphi)g - \varphi^* h \circ \chi_{1}(\varphi),
\end{equation*}
is the $\sigma_2$ -- \textit{stress-energy tensor}, cf. \cite[p. 44]{cri}.
\end{re}

\subsection{Consequences of the first variation formula}

Let us take a look firstly to the simplest (non-trivial) case, namely $\mathrm{dim}N=2$, so that $\varphi$ will have (at most) two distinct eigenvalues.
\begin{co} \label{2di}
Let $\varphi: (M^m,g) \to (N^2,h)$ be a submersion taking values in a surface $(m\geq 2)$. Let $\lambda_{1,2}: M \to (0, \infty)$ denote the (positive) square roots of eigenvalues of its Cauchy-Green tensor and let $\mu^{\VV}$ denote the mean curvature vector field of its fibers. Then $\varphi$ is $\sigma_2$-critical if and only if the following equation is satisfied:
\begin{equation}\label{fh}
\gr^{\HH}(\ln \lambda_{1}\lambda_{2}) - (m-2)\mu^{\VV}=0.
\end{equation}
In particular a (local) diffeomorphism $\varphi: (M^2,g) \to (N^2,h)$ is 
$\sigma_2$-critical if and only if $\lambda_{1}\lambda_{2} \equiv const.$, i.e. it preserves areas up to a constant factor. 
\end{co}

\begin{proof}
Let $\{ E_1 , E_2 , E_{\gamma} \} _{\gamma = 1, 2, ..., m-2}$ a local orthonormal frame of eigenvector fields for $\varphi^* h$, corresponding to the eigenvalues $\lambda_{1}^{2}$, $\lambda_{2}^{2}$ and 0, respectively (i.e. $E_{\gamma}$'s span $\VV = \mathrm{Ker} \dif \varphi$). Since $\dif \varphi (E_1)$ and $\dif \varphi (E_2)$ are orthogonal, $\varphi$ will be $\sigma_2$--critical iff:
\begin{equation}\label{sufcond}
h(\tau_{\sigma_2}(\varphi), \dif \varphi (E_1))=0 \quad \text{and} \quad 
h(\tau_{\sigma_2}(\varphi), \dif \varphi (E_2))=0.
\end{equation}

An easy simplification shows us that:
\begin{equation*}
\begin{split}
\tau_{\sigma_2}(\varphi) =& 
\lambda_{1}^{2}\nabla \dif \varphi(E_2, E_2) + 
\lambda_{2}^{2}\nabla \dif \varphi(E_1, E_1) -
(m-2)(\lambda_{1}^{2} + \lambda_{2}^{2})\dif \varphi \left( \mu^{\VV} \right)\\
& + \dif \varphi \left( \gr(\lambda_{1}^{2} + \lambda_{2}^{2}) \right)
- \dif \varphi \left( [\di \varphi^* h]^\sharp \right).
\end{split}
\end{equation*}

Recall that, according to \cite{slub}, we have:
\begin{align*}
\left(\mathrm{div}\varphi^{*}h\right)(E_{1}) &= E_{1}(\lambda_{1}^{2})+(\lambda_{2}^{2}-\lambda_{1}^{2})g(\nabla_{E_{2}}E_{2}, E_{1}) - (m-2)\lambda_{1}^{2}g(\mu^{\VV}, E_{1});\\
\left(\mathrm{div}\varphi^{*}h\right)(E_{2}) &= E_{2}(\lambda_{2}^{2})+(\lambda_{1}^{2}-\lambda_{2}^{2})g(\nabla_{E_{1}}E_{1}, E_{2}) - (m-2)\lambda_{2}^{2}g(\mu^{\VV}, E_{2})
\end{align*}
and also, according to \cite[Lemma 1]{slub}:
\begin{align*}
h(\nabla \dif \varphi (E_{k},E_{k}), \dif \varphi (E_{k})) &= 
\frac{1}{2} E_{k}(\lambda_{k}^{2}), \quad \forall k\\
h( \nabla \dif \varphi (E_{i}, E_{i}), \dif \varphi (E_{k})) &= -\frac{1}{2} E_{k}(\lambda_{i}^{2}) +
(\lambda_{i}^{2}-\lambda_{k}^{2})g(\nabla_{E_{i}} E_{i}, \, E_{k}), \quad \forall i \neq k. 
\end{align*}

Putting all together, we translate \eqref{sufcond} as following: 
\begin{equation}\label{sil}
\left\{
\begin{array}{ccc}
\frac{1}{2}\lambda_{1}^{2}E_1 (\lambda_{2}^{2}) +
\frac{1}{2}\lambda_{2}^{2}E_1 (\lambda_{1}^{2}) - 
(m-2)\lambda_{1}^{2}\lambda_{2}^{2}g(\mu^{\VV}, E_1) &=& 0\\[3mm]
\frac{1}{2}\lambda_{1}^{2}E_2 (\lambda_{2}^{2}) +
\frac{1}{2}\lambda_{2}^{2}E_2 (\lambda_{1}^{2}) - 
(m-2)\lambda_{1}^{2}\lambda_{2}^{2}g(\mu^{\VV}, E_2) &=& 0
\end{array}
\right.
\end{equation}
from which the result follows. 
\end{proof}

Before letting the geometry coming in, it worth to see a simple example in the flat case (note that the notion of $\sigma_2$-critical map can be extended also to noncompact domains by imposing the first variation to be zero on any compact sub-domain).
\begin{ex}\label{hen}
The H\'enon map \cite{hen} $H: \RR^2 \to \RR^2$ is given by $H(x, y)=(y+1-ax^2 , bx)$, where $a,b$ are real parameters and $b\neq 0$.
It is easy to check that in this case we have:
$\lambda_{1}^{2}\lambda_{2}^{2} = det(J_{H}^{t}J_H)=b^2$, so $H$ is a $\sigma_2$-critical map.\\
Let us notice that $H$ is moreover harmonic/holomorphic if and only if $a=0$ and $b=1$, that is when it is an isometry.
\end{ex}
More examples of area-preserving maps (up to constants) are to be found in \cite{mac}. An elementary example in non-flat case, is given by the map between 2-spheres $\left(\cos u , \ \sin u \cdot e^{\ii v}\right) \mapsto \left(\cos u , \ \sin u \cdot e^{\ii kv}\right)$, where $k$ is an integer that gives the degree of the map.

Example \ref{hen} reflects a more general fact, as we can easily check the following:
\begin{re}
A holomorphic map $f: D \subset \CC \to \CC$ is $\sigma_2$-critical if and only if it is homothetic. 
\end{re}
This remark suggests us that it might be difficult to find topologically interesting mappings that are both harmonic and $\sigma_2$-critical. However, this is the case of the identity map between 3-spheres and standard Hopf map, the only exact solutions known for Skyrme and Faddeev models.

The next corollary follows in full general context the interplay between harmonicity/holomorphicity and $\sigma_2$-criticality. Compare it with the results in \cite{slo} referring to another generalization of the Faddeev model.

\begin{co}\label{crit}

\noindent $(i)$ Any \emph{totally geodesic map} is $\sigma_2$--critical.

\medskip
\noindent $(ii)$ A \emph{harmonic map} $\varphi$ is $\sigma_2$--critical if and only if, at any point it satisfies:
\begin{equation}\label{ELha}
\dif \varphi(\mathrm{grad}e(\varphi))=
\sum_i \lambda_{i}^{2} \nabla \dif \varphi (e_i , e_i),
\end{equation}
where $\{e_i\}$ is a local orthonormal frame of eigenvectors for $\varphi^* h$ around that point.

\medskip
\noindent $(iii)$ A \emph{pseudo horizontally homothetic harmonic map} $\varphi$ to a K\"ahler manifold is $\sigma_2$--critical if and only if \ $\gr e(\varphi) \in \Ker \dif \varphi$. \\
In particular, any \emph{holomorphic map} between K\"ahler manifolds (or $(\phi,J)$-holomorphic from Sasaki to K\"ahler) which has constant Dirichlet energy density, is also $\sigma_2$--critical. 
\end{co}

\begin{proof} $(i)$ By definition, a totally geodesic map satisfies $\nabla \dif \varphi=0$. In this case it is known that $\varphi ^* h$ is parallel and its eigenvalues are constant. Consequently every term in \eqref{EL} cancels.

\medskip

\noindent $(ii)$ Recall that, for any smooth map we have 
the identity (cf. \cite[Lemma 3.4.5]{ud}):
\begin{equation}\label{stres}
\mathrm{div}S({\varphi})= \dif e(\varphi) - \mathrm{div}\varphi^*h =
-h(\tau(\varphi), \dif \varphi),
\end{equation}
where $S({\varphi}):= e(\varphi)g - \varphi^*h$ is the \textit{stress-energy tensor} of the map.

In particular, for a harmonic map we have $\tau(\varphi)=0$, so $\dif \varphi(\mathrm{grad}e(\varphi))=\dif \varphi(\di \Cc_\varphi)$, relation that simplifies \eqref{EL} to \eqref{ELha}.

\medskip

\noindent $(iii)$ As for any PHWC mapping the eigenvalues of $\varphi^* h$ are double, according to $(ii)$ a PHH harmonic map must satisfy:
\begin{equation*}
\dif \varphi(\mathrm{grad}e(\varphi))=
\sum_i \lambda_{i}^{2} [\nabla \dif \varphi (e_i , e_i) + \nabla \dif \varphi (F^{\varphi}e_i , F^{\varphi}e_i)].
\end{equation*}
But PHH hypothesis assures precisely that $\nabla \dif \varphi (X , X) + \nabla \dif \varphi (F^{\varphi}X , F^{\varphi}X)=0$, $\forall X \in (\Ker \dif \varphi) ^{\perp}$. Then our conclusion easily follows.
\end{proof}
Recall that a map that satisfies $\gr e(\varphi) \in \Ker \dif \varphi$ is called $\infty$-\textit{harmonic}. For more details and examples see \cite{ou}. Notice also that when $M=\Ss^3$ and $N=\Ss^2$, a map that satisfies $(iii)$ must be a harmonic morphism with constant dilation; so  essentially it is the Hopf map, according to \cite{bai}. But Hopf map is a particular case of Boothby-Wang fibration as it will be pointed out below.

\medskip
As for the interplay between conformality and $\sigma_2$-critical maps, we have (compare again with \cite{slo}):

\begin{co}\label{crit+} Let $\varphi: (M^m,g) \to (N^n,h)$ be a \emph{horizontally conformal submersion} with dilation $\lambda$. Let $\mu^{\VV}$ denote the mean curvature vector field of its fibers. Then $\varphi$ is $\sigma_2$--critical if and only if it is $4$-harmonic, that is:
\begin{equation}\label{4morph}
(n-4)\mathrm{grad}^{\HH}(\mathrm{ln} \lambda)+(m-n)\mu^{\VV}=0.
\end{equation}
In particular,

\noindent $(i)$ a \emph{conformal diffeomorphism} is  $\sigma_2$--critical if and only if it is homothetic. 

\medskip
\noindent $(ii)$ a \emph{horizontally homothetic submersion} is  $\sigma_2$--critical if and only if it has minimal fibres. 

\medskip
\noindent $(iii)$ a \emph{horizontally conformal submersion onto a four-manifold} is $\sigma_2$--critical if and only if it has minimal fibres.
\end{co}
\begin{proof}
For a HC submersion we have $\Cc_{\varphi} \vert _{\HH}=\lambda^2 Id$, where $\HH$ is the horizontal distribution. So the terms involving the Cauchy-Green tensor in \eqref{EL} are equal to
$$
\tr (\nabla \dif \varphi) \circ \Cc_\varphi  
+\dif \varphi(\di \Cc_\varphi)= \tr \nabla (\dif \varphi \circ \Cc_\varphi)=\lambda^2 \tau(\varphi) + \dif \varphi(\gr \lambda^2).
$$
Recall that for HC submersions of dilation $\lambda$ the tension field is given by \cite[Prop. 4.5.3]{ud}:
$$
\tau(\varphi)=-\dif \varphi \left((n-2)\gr \ln \lambda + (m-n)\mu^{\VV}\right).
$$
Replacing the two above identities in \eqref{EL} and taking into account that $e(\varphi)=(n/2)\lambda^2$ we get the equation \eqref{4morph}. \end{proof}
Recall that a horizontally conformal submersion that satisfies \eqref{4morph} is called $4$-\textit{harmonic morphism}. For more details and examples see \cite{bur}.

Let us point out the following class of examples for both Corollary \ref{crit} $(iii)$ and Corollary \ref{crit+} $(ii)$:
\begin{ex}\label{BWfibr}
A Boothby-Wang fibration (see \cite{ble}) of a compact, regular contact manifold $(M, \eta)$ over a K\"ahler (or just almost K\"ahler) manifold $(N, J, h)$ is a harmonic $\sigma_2$-critical map, as the total space is endowed with the metric $g=\varphi^* h + \eta \otimes \eta$, so that the fibration is a Riemannian submersion (so $e(\varphi)$ is constant) with minimal fibers. Note that all Hopf fibrations belong to this class of examples.
\end{ex}

\medskip

As the Skyrme model deals with maps taking values in the 3-sphere, let us particularize our Euler-Lagrange equations to the 3-dimensional target case. Analogously to the Corollary \ref{2di}, we can establish:

\begin{co} 
Let $\varphi: (M^m,g)\to (N^3,h)$ be a submersion to a three-manifold. Let $\lambda_{1}^{2}$, $\lambda_{2}^{2}$, $\lambda_{3}^{2}$ be the (non-zero) eigenvalues of its Cauchy-Green tensor and $\mu^{\VV}$ the mean curvature field of its fibers. Then $\varphi$ is $\sigma_2$-critical if and only if:
\begin{equation}\label{sig3}
\left\{
\begin{array}{ccc}
\frac{1}{2}E_1 (\lambda_{1}^{2}\lambda_{2}^{2} +\lambda_{1}^{2}\lambda_{3}^{2} - \lambda_{2}^{2}\lambda_{3}^{2}) +
\lambda_{3}^{2}(\lambda_{2}^{2} - \lambda_{1}^{2})\Gamma_{22}^{1} + \lambda_{2}^{2}(\lambda_{3}^{2}-\lambda_{1}^{2})\Gamma_{33}^{1} -
(m-3)\lambda_{1}^{2}(\lambda_{2}^{2} + \lambda_{3}^{2}) g(\mu^{\VV}, E_1)&=&0\\[3mm]
\frac{1}{2}E_2 (\lambda_{1}^{2}\lambda_{2}^{2} + \lambda_{2}^{2}\lambda_{3}^{2} - \lambda_{1}^{2}\lambda_{3}^{2}) +
\lambda_{1}^{2}(\lambda_{3}^{2}-\lambda_{2}^{2})\Gamma_{33}^{2} +
\lambda_{3}^{2}(\lambda_{1}^{2}-\lambda_{2}^{2})\Gamma_{11}^{2} - (m-3)\lambda_{2}^{2}(\lambda_{1}^{2} + \lambda_{3}^{2}) g(\mu^{\VV}, E_2)&=&0\\[3mm]
\frac{1}{2}E_3 (\lambda_{2}^{2}\lambda_{3}^{2} +\lambda_{1}^{2}\lambda_{3}^{2} - \lambda_{1}^{2}\lambda_{2}^{2}) +
\lambda_{2}^{2}(\lambda_{1}^{2}-\lambda_{3}^{2})\Gamma_{11}^{3} + \lambda_{1}^{2}(\lambda_{2}^{2}-\lambda_{3}^{2})\Gamma_{22}^{3} -
(m-3)\lambda_{3}^{2}(\lambda_{1}^{2} + \lambda_{2}^{2}) g(\mu^{\VV}, E_3)&=&0
\end{array}
\right.
\end{equation}
where $\Gamma_{ij}^{k} := g(\nabla_{E_i}E_j, E_k)$.
\end{co}

In particular, if $\varphi$ is horizontally conformal, i.e. $\lambda_{1}^{2} = \lambda_{2}^{2} = \lambda_{3}^{2}$,
notice that \eqref{sig3} is equivalent to the equation of 4-harmonicity \eqref{4morph}.

As suggested by the elementary example of H\'enon map, the most natural choice in finding $\sigma_2$-critical maps is given by area preserving (up to a constant rescaling) maps. Supposing $m=3$ and  $\lambda_{1}^{2} = \lambda_{2}^{2} \cdot \lambda_{3}^{2}=k^2$, after simplification of a $k^2$ factor in every equation, \eqref{sig3} becomes (with $\lambda_{2}^{2}=\lambda^{2}$):

\begin{equation}\label{contactsig3}
\left\{
\begin{array}{ccc}
\frac{1}{2}E_1 (\lambda^{2} + \frac{k^2}{\lambda^{2}}) +
\frac{1}{\lambda^{2}}(\lambda^{2} - k^{2})\Gamma_{22}^{1} + \lambda^{2}(\frac{1}{\lambda^{2}} -1)\Gamma_{33}^{1} &=&0\\[3mm]
\frac{1}{2}E_2 (\lambda^{2} - \frac{k^2}{\lambda^{2}}) +
(\frac{k^2}{\lambda^{2}} -\lambda^{2})\Gamma_{33}^{2} +
\frac{1}{\lambda^{2}}(k^{2}-\lambda^{2})\Gamma_{11}^{2}&=&0\\[3mm]
\frac{1}{2}E_3 (\frac{k^2}{\lambda^{2}} - \lambda^{2}) +
\lambda^{2}(1-\frac{1}{\lambda^{2}})\Gamma_{11}^{3} + (\lambda^{2}-\frac{k^2}{\lambda^{2}})\Gamma_{22}^{3} &=&0
\end{array}
\right.
\end{equation}

If moreover the eigenvector corresponding to $\lambda_{1}^{2}$ is a complete Killing vector field on $M$, then $\Gamma_{11}^{i} =0$ (its integral curves provide a minimal foliation on $M$) and $\Gamma_{ii}^{1} =0$ (the complementary distribution is totally geodesic). The above system simplifies once more and, after comparing with the analogous system of equations for harmonicity of $\varphi$ (\cite[(7)]{slub}):
\begin{equation}\label{harm-gen}
E_{k}[\lambda_{k}^{2}- e(\varphi)]+\sum_{i=1}^{3}(\lambda_{i}^{2}- \lambda_{k}^{2}) \Gamma_{ii}^{k} -(m-3) \lambda_{k}^{2} g(\mu^{\VV}, E_{k})=0, \quad k=1,2,3
\end{equation}
we can conclude as follows:
\begin{te}\label{contactomo}
Let $\varphi$ be a smooth mapping between $3$-dimensional Riemannian manifolds. Suppose that the eigenvalues of its Cauchy-Green tensor satisfy $\lambda_{1}^{2} = \lambda_{2}^{2} \cdot \lambda_{3}^{2}=k^2$,
$k\in \RR$ and that the eigenvector corresponding to $\lambda_{1}^{2}$ is a (globally defined) Killing vector field. Then the following statements are equivalent:

$(a)$ $\varphi$ is harmonic 

$(b)$ $\varphi$ is $\sigma_2$--critical

$(c)$ $\frac{1}{2} \gr(\lambda_{2}^{2} - \lambda_{3}^{2})=(\lambda_{2}^{2} - \lambda_{3}^{2}) \left(\nabla_{E_2}E_2 + \nabla_{E_3}E_3\right)$.

In particular, if moreover $\lambda_{2}^{2} = \lambda_{3}^{2}=k$, then $\varphi$ is both harmonic and $\sigma_2$--critical, so critical configuration for the full Skyrme model with arbitrary coupling constant.
\end{te}

As we mentioned in the previous section a particular class of maps that satisfy the hypothesis $\lambda_{1}^{2} = \lambda_{2}^{2} \cdot \lambda_{3}^{2}=k^2$ is provided by contactomorphisms. So, in particular, we have:
\begin{co}
A contactomorphism $\varphi: (M^3, \eta, g) \to (N^3, \eta^{\prime}, g^{\prime})$ from a $K$-contact 3-manifold to a contact metric 3-manifold such that $\varphi^* \eta^{\prime} = k \cdot \eta$, $k\in \RR$ is $\sigma_{1,2}$--critical if and only if the equation $(c)$ is verified.

Moreover, if $M$ and $N$ are closed manifolds, then 
$\mathrm{deg}\varphi = k^2\frac{\mathrm{Vol}(M, g)}{\mathrm{Vol}(N, g^{\prime})}$.
\end{co}

See Example \ref{tact}$(i)$ below for an illustration of the above phenomenon on the Heisenberg group.
\subsection{The effect of (bi)conformal changes of metric}
Let $\varphi: (M^m,g) \to (N^n,h)$ be an almost submersion and denote as usual $\Ker \dif \varphi = \VV$, $\HH=\VV^{\perp}$. 
Recall \cite{ud, Pantilie} that under a \textit{biconformal change} of metric,
$\ov{g}=\sigma^{-2}g^{\HH}+\rho^{-2}g^{\VV}$, the tension field of $\varphi$ becomes
\begin{equation}
\ov \tau(\varphi) = \sigma^2\{\tau(\varphi) + \dif \varphi\left(\gr \ln (\sigma^{2-n}\rho^{n-m})\right)\},
\end{equation}
where $\sigma, \rho$ are nowhere vanishing functions on $M$. 

Let $\gr_{\varphi^* h}f = \sum_k \lambda_{k}^{2}E_k(f)E_k$ denote the \textit{weighted gradient} (with respect to the Cauchy-Green tensor of $\varphi$) of a function $f$ defined on $M$, where $E_k$ are eigenvectors of $\varphi^* h$ corresponding to the eigenvalues $\lambda_{k}^{2}$.

Using a standard technique, similar to the harmonic case we can state the following:
\begin{te}\label{bic}
Under a biconformal change of metric,
the $\sigma_2$-tension field of $\varphi$ becomes
\begin{equation}\label{biconf}
\ov \tau_{\sigma_2}(\varphi) = \sigma^4\{\tau_{\sigma_2}(\varphi) + 2e(\varphi)\dif \varphi\left(\gr \ln (\sigma^{4-n}\rho^{n-m})\right) - \dif \varphi\left(\gr_{\varphi^* h} \ln (\sigma^{4-n}\rho^{n-m})\right)\}
\end{equation}
In particular, taking $\sigma =\rho$, we have
the corresponding formula for conformal changes of metric:
\begin{equation*}
\ov \tau_{\sigma_2}(\varphi) = \sigma^4\{\tau_{\sigma_2}(\varphi) + 2e(\varphi)\dif \varphi\left(\gr \ln (\sigma^{4-m})\right) - \dif \varphi\left(\gr_{\varphi^* h} \ln (\sigma^{4-m})\right)\}
\end{equation*}

\end{te}

\begin{co}
\begin{enumerate}
\item[(i)] When $m \neq n$, $\varphi$ is $\sigma_2$-critical with respect to $g$ if and only if is $\sigma_2$-critical with respect to 
$\ov g=\sigma^{-2}g^{\HH}+\sigma^{\frac{2(n-4)}{m-n}}g^{\VV}$.

\item[(ii)] if $m=3$, $n=2$, then $\sigma_2$-Euler-Lagrange equations are invariant under the change of metric $\ov{g}=\sigma^{-2}g^{\HH}+\sigma^{-4}g^{\VV}$, i.e. for and $\sigma^2=\rho$;


\item[(iii)] if $m=4$, then $\sigma_2$-Euler-Lagrange equations are invariant under conformal changes of metric, $\ov{g}=\sigma^{-2}g$, i.e. for $\sigma=\rho$;

\item[(iv)] if $\varphi$ is a horizontally conformal map with dilation $\lambda$, then \eqref{biconf} simplifies to:
\begin{equation*}
\ov \tau_{\sigma_2}(\varphi) = \sigma^4\{\tau_{\sigma_2}(\varphi) + (n-1)\lambda^2\dif \varphi\left(\gr \ln (\sigma^{4-n}\rho^{n-m})\right) \}.
\end{equation*}
\end{enumerate}
\end{co}

The above formulae for the effect of changes of domain metric can be used to construct $\sigma_{2}$ or $\sigma_{1,2}$-critical maps with respect to a metric (bi)conformally related to a given "standard" one. The invariance results can be used to construct a $\sigma_{1,2}$-critical map with respect to a (bi)conformally related metric $\ov g$ starting with a $\sigma_{2}$-solution with respect to $g$.

\section{Weak $\sigma_2$-Stability}

Let $\{\varphi_{t,s}\}$ a (smooth) two-parameter variation of $\varphi$ with variation vector fields 
$v,w \in \Gamma(\varphi^{-1}TN)$, i.e.
$$
v(x) = \displaystyle{\frac{\partial \varphi_{t,s}}{\partial t}(x) \Bigl\lvert _{(t,s)=(0,0)}}, \qquad
w(x) = \displaystyle{\frac{\partial \varphi_{t,s}}{\partial s}(x) \Bigl\lvert _{(t,s)=(0,0)}}, \qquad \forall x \in M.
$$ 

We ask when the following bilinear function is positive semi-definite for a $\sigma_2$-critical mapping $\varphi$, which will be consequently called \textit{(weakly) stable}:
\begin{equation*}
\mathrm{Hess}^{\sigma_2}_{\varphi}(v,w) = \displaystyle{\frac{\partial^{2}}{\partial t \partial s}\mathcal{E}_{\sigma_2}(\varphi_{t,s}) \Bigl\lvert _{(t,s)=(0,0)}}
\end{equation*}

Let us now recall some standard notations: $\langle \cdot, \cdot \rangle$ is the inner product on $\otimes ^p T^* M$ induced from $g$ in the standard way (and $\abs{\cdot}$ is the corresponding norm); \ $\mathrm{Ric}^{\varphi}$ is the fiberwise linear bundle map on $\varphi^{-1}TN$ defined by $\mathrm{Ric}^{\varphi} v = \tr R^N(v, \dif \varphi)\dif \varphi$; \ $(\nabla^{\varphi})^2$ is the second order operator on $\Gamma(\varphi^{-1}TN)$ defined as $[(\nabla^{\varphi})^2 v](X,Y) =\nabla^{\varphi}_X \nabla^{\varphi}_Y v-\nabla^{\varphi}_{\nabla_X Y} v$; \ $\Delta^{\varphi}=\tr (\nabla^{\varphi})^2$ is the rough Laplacian along $\varphi$ that has the property: $\int _M h(\Delta^{\varphi} v, v) \nu_g = 
-\int _M \langle \nabla^{\varphi} v, \nabla^{\varphi} v \rangle \nu_g$ on compactly supported sections.

\begin{pr}[The second variation formula]
\begin{equation}\label{sec}
\begin{split}
\frac{\partial^{2}}{\partial t \partial s} \mathcal{E}_{\sigma_2}(\varphi_{t,s}) \Bigl\lvert _{(0,0)}
=& 2\int_M \left\{\di^{\varphi}v \cdot \di^{\varphi}w 
+e(\varphi)\left[ \langle \nabla^{\varphi} v, \nabla^{\varphi} w \rangle - h(\mathrm{Ric}^{\varphi}v, w) \right]\right\}\nu _g\\
&+\int_M  \left\{2\langle \alpha_v, h \left( w, \nabla \dif \varphi \right) \rangle +
\langle h \left((\nabla^{\varphi})^{2}v + R^N(v, \dif \varphi)\dif \varphi, \ w \right), \varphi^*h \rangle \right\} \nu _g\\
&+\int_M \left\{
X_w(\mathrm{div}X_v)+
h(\tr (\nabla^{\varphi})^{2}v + \mathrm{Ric}^{\varphi}v, \dif \varphi(X_w))\right\} \nu _g\\
&+\int_M \left\{-h\left(\nabla^{\varphi}_{X_w}\tau(\varphi), v\right)+
h\left( w, \nabla^{\varphi}_{\mathrm{div}\Cc_{\varphi}}v \right) \right\} \nu _g.
\end{split}
\end{equation}
The term $2\langle \alpha_v, h \left( w, \nabla \dif \varphi \right) \rangle$ can be also written as 
$\langle \mathcal{L}_{X_v}g, h \left(w, \nabla \dif \varphi \right) \rangle - 2\langle h \left( v, \nabla \dif \varphi \right), h \left( w, \nabla \dif \varphi \right) \rangle$.
\end{pr}

\begin{proof}
We have:
\begin{equation*}
\frac{\partial^{2}}{\partial t \partial s} \mathcal{E}_{\sigma_2}(\varphi_{t,s})=
- \int _M \left\{ h\left( \nabla_{\partial / \partial t}^{\Phi} \frac{\partial \Phi}{\partial s}, \tau_{\sigma_2}(\varphi_{t,s}) \right) +  h\left( \frac{\partial \Phi}{\partial s}, 
\nabla_{\partial / \partial t}^{\Phi} \tau_{\sigma_2}(\varphi_{t,s}) \right) \right\} \nu_g,
\end{equation*}
where $\tau_{\sigma_2}(\varphi)= \tau_{4}(\varphi)-
\tr (\nabla \dif \varphi) \circ \Cc_{\varphi} - 
\dif \varphi(\di \Cc_{\varphi})$ is the Euler-Lagrange operator calculated in the previous section and $\tau_{4}(\cdot)$ is the 4-tension field, cf. \eqref{21}.

The first line in \eqref{sec} is derived from $\tau_{4}(\varphi_{t,s})$ term, cf. \cite{take} (for a detailed proof see \cite{ara}).

Let $\{e_i\}_{i=1,...,m}$ be a local orthonormal frame on $M$. 

The variation of the term $\tr (\nabla \dif \varphi) \circ \Cc_{\varphi}$ gives us:
\begin{equation*}
\begin{split}
&h\left(\frac{\partial \Phi}{\partial s}, 
\nabla_{\partial / \partial t}^{\Phi} \left[\sum_i \varphi_{t,s}^* h(e_i, e_i) \nabla \dif \varphi_{t,s} (e_i , e_i)
+2 \sum_{i<j} \varphi_{t,s}^* h(e_i, e_j) \nabla \dif \varphi_{t,s} (e_i , e_j)\right] \right) \Bigl\lvert _{(t,s)=(0,0)} \\
&= \sum_i 2\alpha_v(e_i, e_i) h(w , \nabla \dif \varphi(e_i,e_i))
+2 \sum_{i<j} [\alpha_v(e_i, e_j)+ \alpha_v(e_j, e_i)] h(w , \nabla \dif \varphi(e_i,e_j)) \\
& \ + \sum_{i}\varphi^* h(e_i, e_i) h\left(w , (\nabla^{\varphi})^{2}_{e_i, e_i}v + R^N(v, \dif \varphi(e_i))\dif \varphi(e_i) \right) \\
& \ + 2 \sum_{i<j}\varphi^* h(e_i, e_j) h\left(w , (\nabla^{\varphi})^{2}_{e_i, e_j}v + R^N(v, \dif \varphi(e_i))\dif \varphi(e_j) \right)\\
&= 2\langle \alpha_v, h \left( w, \nabla \dif \varphi \right) \rangle +
\langle h \left((\nabla^{\varphi})^{2}v + R^N(v, \dif \varphi)\dif \varphi, \ w \right), \varphi^*h \rangle.
\end{split}
\end{equation*}

The variation of the term $\dif \varphi(\di \Cc_{\varphi})$ gives us:
\begin{equation*}
\begin{split}
&h\left( \frac{\partial \Phi}{\partial s}, 
\nabla_{\partial / \partial t}^{\Phi} 
\left[(\mathrm{div}\varphi_{t,s}^{*}h)(e_j) \, \dif \varphi_{t,s}(e_j) \right] \right) \Bigl\lvert _{(t,s)=(0,0)} \\
&=h\left( \frac{\partial \Phi}{\partial s}, 
\nabla_{\partial / \partial t}^{\Phi} 
\left[e_j(e(\varphi_{t,s})) +
h(\tau(\varphi_{t,s}), \dif\varphi_{t,s}(e_j) \right]\, \dif \varphi_{t,s}(e_j) \right) \Bigl\lvert _{(t,s)=(0,0)} \\
&=e_j\left[\mathrm{div}X_v - h(\tau(\varphi),v)\right]
h(w, \dif \varphi(e_j)) +
h(w , \nabla^{\varphi}_{\mathrm{grad} e(\varphi)}v)\\
& \ + h(\tr(\nabla^{\varphi})^{2}v + \mathrm{Ric}^{\varphi}v, \dif \varphi(e_j))h(w, \dif \varphi(e_j))\\
& \ +h(\tau(\varphi),\nabla^{\varphi}_{e_j}v)
h(w, \dif \varphi(e_j))-
h(w , \nabla^{\varphi}_{[\mathrm{div}S({\varphi})]^{\sharp}}v)\\
&=X_w(\mathrm{div}X_v)+
h(\tr (\nabla^{\varphi})^{2}v + \mathrm{Ric}^{\varphi}v, \dif \varphi(X_w))-h\left(\nabla^{\varphi}_{X_w}\tau(\varphi), v\right)+
h\left( w, \nabla^{\varphi}_{[\mathrm{div}\varphi^*h]^{\sharp}}v \right),
\end{split}
\end{equation*}
where all repeated indices are summed and we have used again \eqref{stres}.
\end{proof}

\begin{re}
Another version of the second variation formula for $\sigma_2$--energy
can be obtained from the general formula derived in \cite[p. 37]{cri}, which has the advantage of revealing the associated $\sigma_p$--Jacobi operator. Nevertheless one of its terms still contains a derivative $\frac{d}{dt}$ so is difficult to use it directly. Here we shall work with \eqref{sec} which has explicit terms.
\end{re}

Let us notice that, according to Remark \ref{31}, we have
\begin{equation*}
\begin{split}
&\di ^{\varphi} v \di ^{\varphi} w =
[\di X_v - h(v, \tau(\varphi))][\di X_w - h(w, \tau(\varphi))]\\
&= 
(\di X_v)(\di X_w) + h(v, \tau(\varphi))h(w, \tau(\varphi))-
 h(w, \tau(\varphi))\di X_v - h(v, \tau(\varphi))\di X_w\\
&= 
\di X_v \di X_w + h(v, \tau(\varphi))h(w, \tau(\varphi))+
h(\nabla^{\varphi}_{X_v}w + \nabla^{\varphi}_{X_w}v, \tau(\varphi))\\
&+ h\left(\nabla^{\varphi}_{X_v}\tau(\varphi), w\right) + h\left(\nabla^{\varphi}_{X_w}\tau(\varphi), v\right)
+ \ \text{divergence terms}
\end{split}
\end{equation*}

Moreover, as the identity $X(\di Y) + \di X \di Y=\di ((\di Y) X)$ always holds, on a closed Riemannian manifold $(M,g)$ we have:
\begin{equation*}
\int_M \left[ X(\di Y) + \di X \di Y \right] \nu _g =0, \quad \forall X,Y \in \Gamma(TM).
\end{equation*}

Therefore, using the above observations, we can rewrite \eqref{sec} in a different form. As the simplifications that occur are not enlightening in the general case, we shall apply them only in particular situations, as we shall see below.

\medskip

In the rest of this section we restrict ourselves to the particularly important case of harmonic maps that satisfy also \eqref{ELha}, i.e.  the case of harmonic $\sigma_2$--critical maps. In particular, these maps are critical points for the \textit{full} energy \eqref{generg}. According to the above observations, in this case \eqref{sec} simplifies  to:

\begin{co}[$\sigma_2$--Hessian of harmonic $\sigma_2$--critical mappings]
\begin{equation}\label{secha}
\begin{split}
\mathrm{Hess}^{\sigma_2}_{\varphi}(v,v)
=& \int_M \left\{ 2e(\varphi)\left[ \abs{\nabla^{\varphi} v}^2 - h\left(\mathrm{Ric}^{\varphi}v, v\right) \right] + (\di X_v)^2 \right\} \nu _g\\
&+\int_M \left\{2\langle \alpha_v, h \left( v, \nabla \dif \varphi \right) \rangle +
\langle h \left((\nabla^{\varphi})^{2}v + R^N(v, \dif \varphi)\dif \varphi, \ v \right), \varphi^*h \rangle \right\} \nu _g\\
&+\int_M \left\{
h\left(\tr (\nabla^{\varphi})^{2}v + \mathrm{Ric}^{\varphi}v, \dif \varphi(X_v) \right) + 
h\left(v, \nabla^{\varphi}_{\gr e(\varphi)}v \right)\right\} 
\nu _g.
\end{split}
\end{equation}
\end{co}

\medskip

As we saw in the Corollary 3.3, one of the simplest examples of harmonic $\sigma_2$-critical mappings is provided by horizontally homothetic (HH) submersions (i.e. semiconformal with vertical gradient of the dilation) that have minimal fibres. In the following, by submersion we mean surjective submersion and $\VV=\Ker \dif \varphi$, $\HH = \VV^{\perp}$.

\begin{co}[$\sigma_2$--Hessian of harmonic HH submersions] \label{coheshh}
\begin{equation}\label{heshh}
\begin{split}
\mathrm{Hess}^{\sigma_2}_{\varphi}(v,v)
=& \int_M \left\{(n-2) \lambda^2 \left[ \abs{\nabla^{\varphi} v}^2 - h(\mathrm{Ric}^{\varphi}v, v) \right] + (\mathrm{div}X_v)^2 
\right\} \nu _g \\
&+\int_M \left\{
-\lambda^2 h\left(\Delta_{\VV}^{\varphi}v, v \right)
+ \frac{n-4}{2}h\left(\nabla^{\varphi}_{\gr ^{\VV}\lambda^2}v, v \right) \right\} \nu _g \\
&+2\int_M h \left( v, \sum_{\gamma=1}^{m-n}\nabla \dif \varphi([e_{\gamma}, X_v]^{\HH}, e_{\gamma}) - 2\nabla \dif \varphi(\gr^{\VV} \ln \lambda, X_v) \right) \nu _g,
\end{split}
\end{equation}
where $\Delta_{\VV}^{\varphi}v=\tr \vert _{\VV} (\nabla^{\varphi})^2$ and $\{e_\gamma \}_{\gamma=1,..., m-n}$ is a local orthonormal frame of $\VV$.\\
If $\mathrm{dim}M=\mathrm{dim}N$, that is for homothetic local diffeomorphisms (and in particular for the identity map $Id: (M,\lambda^2 g) \to (M,g)$ with $\lambda \equiv$ constant), the above formula takes the simple form:
\begin{equation}\label{identhh}
\mathrm{Hess}^{\sigma_2}_{\varphi}(v,v)
= \int_M \left\{(n-2) \lambda^2 \left[ \abs{\nabla^{\varphi} v}^2 - h(\mathrm{Ric}^{\varphi}v, v) \right] + (\mathrm{div}X_v)^2 
\right\} \nu _g.
\end{equation}
\end{co}

\begin{proof}
As $\varphi^*h \vert _{\HH \times \HH} = \lambda^2 g$,
we have:
\begin{equation*}
\langle h \left((\nabla^{\varphi})^{2}v + R^N(v, \dif \varphi)\dif \varphi, \ v \right), \varphi^*h \rangle=
\lambda^2 h \left(\Delta^{\varphi}v - \Delta_{\VV}^{\varphi}v + \mathrm{Ric}^{\varphi}v, \ v \right).
\end{equation*}

As $v=\lambda^{-2}\dif \varphi(X_v)$, we have:
\begin{equation*}
h\left(\tr (\nabla^{\varphi})^{2}v + \mathrm{Ric}^{\varphi}v, \dif \varphi(X_v) \right) =
\lambda^2 h \left(\Delta^{\varphi}v + \mathrm{Ric}^{\varphi}v, \ v \right).
\end{equation*}

Now let $\{e_i, e_\gamma \}_{i=1,...,n}^{\gamma=1,..., m-n}$ be a local adapted orthonormal frame on $M$ (i.e. $e_\gamma$ span $\VV$ and $e_i$ span the horizontal distribution $\HH = \VV^{\perp}$). As $\nabla \dif \varphi \vert _{\HH \times \HH}=0$ and $\alpha_v \vert _{TM \times \VV}=0$, we have:
\begin{equation*}
\begin{split}
2\langle \alpha_v, h \left( v, \nabla \dif \varphi \right) \rangle &=
2\sum_{i=1}^{n}\sum_{\gamma=1}^{m-n} \alpha_v(e_{\gamma}, e_i) h \left( v, \nabla \dif \varphi(e_{\gamma}, e_i) \right) \\
&=
2\sum_{i=1}^{n}\sum_{\gamma=1}^{m-n} h(\nabla^{\varphi}_{e_{\gamma}}\lambda^{-2}\dif \varphi(X_v), \dif \varphi(e_i)) h \left( v, \nabla \dif \varphi(e_{\gamma}, e_i) \right)\\
&=
2h \left( v, \sum_{\gamma =1}^{m-n}\nabla \dif \varphi([e_{\gamma}, X_v]^{\HH}, e_{\gamma}) -2\nabla \dif \varphi(\gr^{\VV} \ln \lambda, X_v) \right).
\end{split}
\end{equation*}

Now the formula \eqref{heshh} follows directly from \eqref{secha} if we notice that:
\begin{equation*}
\int_M \lambda^2 h \left(\Delta^{\varphi}v, v\right)\nu_g=-
\int_M \left\{\lambda^2 \abs{\nabla^{\varphi}v}^2 +h\left(\nabla^{\varphi}_{\gr ^{\VV}\lambda^2}v, v \right)\right\}\nu_g.
\end{equation*}
\end{proof}

Now we can prove a first stability result:
\begin{pr}\label{sig2stable}
\noindent $(i)$ If $\mathrm{dim} M \in \{2, 3, 4 \}$, then a homothetic local diffeomorphism defined on $M$ is a stable $\sigma_2$-critical map; 

\noindent $(ii)$ If a homothetic local diffeomorphism is a stable map for the Dirichlet energy, then it is a stable $\sigma_2$-critical map too. 
\end{pr}

\begin{proof}
$(i)$ \ As in this case $v=\lambda^{-2}\dif \varphi (X_v)$, we can check that:
\begin{equation*}
\begin{split}
\abs{\nabla^{\varphi} v}^2 - h(\mathrm{Ric}^{\varphi}v, v) 
=\lambda ^{-2}\left(\abs{\nabla X_v}^2 - \mathrm{Ric}^{M}(X_v, X_v)\right). 
\end{split}
\end{equation*}

Therefore, according to \eqref{identhh} we have:
\begin{equation}\label{hesig2}
\begin{split}
\mathrm{Hess}^{\sigma_{2}}_{\varphi}(v,v)
=& \int_M \left\{(n-2) \left[ \abs{\nabla X_v}^2 - \mathrm{Ric}^{M}(X_v, X_v) \right] + (\mathrm{div}X_v)^2 \right\} \nu _g.
\end{split}
\end{equation}

Employing now Yano's identity \cite{yan}
\begin{equation*}
\int_M \left\{\abs{\nabla X}^2 - \mathrm{Ric}(X, X) + (\mathrm{div}X)^2  - \frac{1}{2}\abs{\mathcal{L}_{X}g}^2 \right\} \nu _g=0,
\end{equation*}
we get
\begin{equation}\label{hsig2}
\begin{split}
\mathrm{Hess}^{\sigma_{2}}_{\varphi}(v,v)
=& \int_{M} \left\{\frac{n-2}{2} \abs{\mathcal{L}_{X_v}g}^2 -(n-3)(\mathrm{div}X_v)^2 \right\} \nu _g.
\end{split}
\end{equation}

Notice now that, according to Newton inequalities, we have 
\begin{equation}\label{newton}
\frac{1}{2}\abs{\mathcal{L}_{X_v}g}^2 \geq 
2\sum_i g(\nabla_{e_i}X_v , e_i)^2 \geq
\frac{2}{n} \left[\sum_i g(\nabla_{e_i}X_v , e_i)\right]^2=
\frac{2}{n}(\di X_v)^2,
\end{equation}
where the equality is reached when $X_v$ is a conformal vector field.

Therefore:
\begin{equation*}
\mathrm{Hess}^{\sigma_{2}}_{\varphi}(v,v) \geq \int_M \left[ \frac{2(n-2)}{n} -(n-3) \right](\di X_v)^2 \nu_g = 
\int_M \frac{(4-n)(n-1)}{n} (\di X_v)^2 \nu_g
\end{equation*}
which is nonnegative if $n\leq 4$.
\end{proof}

\begin{re}[HH submersions with 2-dim. target]\label{cohh}
Formula \eqref{heshh} indicates us that HH submersions with 2-dimensional target and  $\lambda \equiv$ constant are certainly privileged maps for the $\mathcal{E}_{\sigma_2}$ variational problem. 

More precisely, let us consider $\varphi: (M^3, g) \to (N^2, h)$ a harmonic homothetic submersion, that is $\gr \lambda=0$ and the fibres are one dimensional and minimal. Then, according to \cite[Prop. 12.3.1]{ud} the fundamental vertical vector (denoted hereafter by $\xi$) of $\varphi$ is Killing. Notice that the term that contains the vertical part of the rough Laplacian along $\varphi$ is given by 
\begin{equation*}
h\left(\Delta_{\VV}^{\varphi}v, v \right)=
\xi h\left(\nabla^{\varphi}_{\xi}v, v \right)
-\norm{\nabla^{\varphi}_{\xi}v}^2 - h\left(\nabla^{\varphi}_{\nabla_{\xi}\xi}v, v \right).
\end{equation*}
But $\nabla_{\xi}\xi=0$ and moreover 
$\int _M \xi h\left(\nabla^{\varphi}_{\xi}v, v \right) \nu_g=0$ because we always have $\int_M X(f) \nu_g= \int_M \di (fX) \nu_g =0$ for a Killing vector field $X$.

Consequently the Hessian of $\varphi$ simplifies to (compare to \cite[(3.10)]{slo}):
\begin{equation}\label{2hh}
\begin{split}
\mathrm{Hess}^{\sigma_2}_{\varphi}(v,v)
&= \int_M \left\{ (\mathrm{div}X_v)^2 + 
\lambda^2 \abs{\nabla^{\varphi} v}_{\VV}^2 + 
2h \left( v, \nabla \dif \varphi([\xi, X_v]^{\HH}, \xi)\right)\right \} \nu_g \\
&= \int_M \left\{ (\mathrm{div}X_v)^2 +  \norm{[\xi, X_v]^{\HH}}^2 + 
2g \left(\nabla_{X_v}\xi, [\xi, X_v]^{\HH}\right) \right \} \nu_g \\
&= \int_M \left\{ (\mathrm{div}X_v)^2 +  \norm{\nabla_{\xi} X_v}^2 - \norm{\nabla_{X_v} \xi}^2 \right \} \nu_g.
\end{split}
\end{equation}

Let us illustrate this situation with the following examples of $\sigma_2$-stability:

\noindent $(a)$ if in addition $\varphi$ has integrable horizontal distribution (i.e. in particular, it is a totally geodesic map and, after a homothetic change of the (co)domain metric, it is locally the projection of a Riemannian product), then it is a stable $\sigma_2$-critical map. This is because the term 
$\nabla \dif \varphi([\xi, X_v]^{\HH}, \xi)$ will be zero in this case.

\medskip
\noindent $(b)$ \textit{the Hopf map} $\mathbb{S}^3 \to \mathbb{S}^2$ between unit spheres with their standard metrics is a stable $\sigma_2$--critical map, according to \cite[Theorem 5.2]{sve} (when $n=2$ the strong coupling limit of the Faddeev-Hopf energy coincides with the $\sigma_2$-energy). It is moreover a minimizer in its homotopy class, cf. \cite{svee}. Notice that this example will be essentially unique, among semiconformal harmonic submersions from $\mathbb{S}^3$ onto a Riemann surface, cf. \cite{bai}.

It is worth to notice that $\sigma_2$-stability of the Hopf map implies its stability for the 4-energy (we know that it is a true minimum, according to \cite{riv}). To see this, take into account Remark 4.4 below to write the 4-Hessian  as 
\begin{equation*}
\begin{split}
\mathrm{Hess}_{\varphi}^{\mathcal{E}_4}(v,v)
&= \int_M (\mathrm{div}X_v)^2 + \frac{1}{2}\abs{\Dd^{\varphi}v}_{\mathcal{H}}^2 + \norm{\nabla^{\varphi}_{\xi}v}^2 - 2 h \left(\nabla^{\varphi}_{\xi} v, Jv \right)\\
&= \int_M (\mathrm{div}X_v)^2 + \frac{1}{2}\abs{\Dd^{\varphi}v}_{\mathcal{H}}^2 + \norm{[\xi, X_v]}^2 + 
2g \left(\nabla_{X_v}\xi, [\xi, X_v]\right),
\end{split}
\end{equation*}
were we use the fact that the Hopf map is a $(\phi,J)$-holomorphic Riemannian submersion from a Sasakian to a K\"ahler manifold.
\end{re}

\medskip

In the end of this section, let us consider the stability of another class of harmonic $\sigma_2$--critical mappings, namely the one given by Corollary \ref{crit}($iii$). To facilitate the exposition consider the simpler case of holomorphic maps between compact K\"ahler manifolds, $\varphi:(M^{2m},J,g)\to (N^{2n},J^N,h)$ (which are in particular PHH harmonic maps).

Define the following connexion in the pull-back bundle, cf. \cite{ura}:
$$
\Dd^{\varphi}v(X):=\nabla^{\varphi}_{JX} v - J^N \nabla^{\varphi}_X v, \quad \forall X\in \Gamma(TM),
$$
that has the immediate property $\Dd^{\varphi}v(JX) + J^N\Dd^{\varphi}v(X)=0, \ \forall X$.

In an orthonormal adapted (local) frame, we can check that:
\begin{equation*}
\begin{split}
&(\nabla^{\varphi})_{e_k, e_k}^{2}v + (\nabla^{\varphi})_{Je_k, Je_k}^{2}v + R^N(v, \dif \varphi(e_k))\dif \varphi(e_k)
+ R^N(v, \dif \varphi(Je_k))\dif \varphi(Je_k) \\
&=J^N \left(\nabla^{\varphi}_{e_k}\Dd^{\varphi}v(e_k) +
\nabla^{\varphi}_{Je_k}\Dd^{\varphi}v(Je_k)-
\nabla^{\varphi}_{\nabla_{e_k}e_k + \nabla_{Je_k}Je_k}\Dd^{\varphi}v\right), \quad \forall k = 1,...,m.
\end{split}
\end{equation*}

From this identity we can deduce that
\begin{equation}\label{dfi}
\begin{split}
&h\left((\nabla^{\varphi})_{e_k, e_k}^{2}v + (\nabla^{\varphi})_{Je_k, Je_k}^{2}v + R^N(v, \dif \varphi(e_k))\dif \varphi(e_k)
+ R^N(v, \dif \varphi(Je_k))\dif \varphi(Je_k) , \ w \right) \\
&=-h\left(\Dd^{\varphi}v(e_k), \ \Dd^{\varphi}w(e_k)\right)- [g(\nabla_{e_k}X_0, e_k) + g(\nabla_{Je_k}X_0, Je_k)], \quad \forall k
= 1,...,m,
\end{split}
\end{equation}
where $X_0$ is defined by $h \left(\Dd^{\varphi}v(Y), J^N w \right):=g(X_0, Y), \ \forall Y$.

\begin{re}
Recall that the Hessian of a harmonic map, for the Dirichlet energy, is given by (see e.g. \cite[p. 92]{ud}):
\begin{equation*}
\mathrm{Hess}_{\varphi}(v,w)= -\int_M h \left(\tr [(\nabla^{\varphi})^{2}v + R^N(v, \dif \varphi)\dif \varphi], w\right)\nu_g \ :=
\int_M h \left(\mathfrak{J}^{\varphi}(v), w \right) \nu_g.
\end{equation*}
For a holomorphic map between compact K\"ahler manifolds, taking the sum in \eqref{dfi} gives us: 
\begin{equation}\label{jachol}
h\left(\tr [(\nabla^{\varphi})^{2}v + R^N(v, \dif \varphi)\dif \varphi], v \right) =
-\frac{1}{2}\abs{\Dd^{\varphi}v}^2 -\di X_1,
\end{equation}
where $X_1$ is defined by $h \left(\Dd^{\varphi}v(Y), J^N v \right):=g(X_1, Y), \ \forall Y$. Therefore $\mathrm{Hess}_{\varphi}(v,v)= \frac{1}{2}\int_M \abs{\Dd^{\varphi}v}^2 \nu_g$ which proves the stability (as harmonic maps) of holomorphic maps between compact K\"ahler manifolds, an infinitesimal version of a classical Lichnerowicz result \cite{ura}.
\end{re}

Now suppose in addition that a holomorphic map between compact K\"ahler manifolds has $\gr e(\varphi) \in \Ker \dif \varphi$. Then it becomes a $\sigma_2$--critical map. By standard techniques, using \eqref{dfi} and a trick similar to \eqref{tric}, we can check the following

\begin{co}[$\sigma_2$--Hessian of holomorphic $\sigma_2$--critical maps between K\"ahler manifolds]
\begin{equation}\label{secho}
\begin{split}
\mathrm{Hess}^{\sigma_2}_{\varphi}(v,v) 
&= \int_M \left\{ (\di X_v)^2 + e(\varphi)\abs{\Dd^{\varphi}v}^2
-\frac{1}{2}\langle \Dd^{\varphi}v, \Dd^{\varphi}v \circ \Cc_{\varphi} \rangle
- \frac{1}{2} \langle \Dd^{\varphi}v, \Dd^{\varphi}\dif \varphi(X_v) \rangle \right\} \nu _g\\
&+\int_M \left\{ 2\langle \alpha_v, h(v, \nabla \dif \varphi)\rangle - h\left(\nabla^{\varphi}_{J\gr ^{\VV}e(\varphi)}v, J^N v \right)\right\} \nu _g ,
\end{split}
\end{equation}
where $\langle \Dd^{\varphi}v, \Dd^{\varphi}v \circ \Cc_{\varphi} \rangle = 2\sum_{k=1}^{m} \lambda_{k}^{2} \norm{\Dd^{\varphi}v(e_k)}^2$.
\end{co}

Let us point our some straightforward consequences of the above formula:

\medskip
\noindent $(i)$ a homothetic local diffeomorphism between compact K\"ahler manifolds is a stable $\sigma_2$--critical map;

\medskip
\noindent $(ii)$ a semiconformal holomorphic map from a compact K\"ahler manifold of non-negative sectional curvature to a K\"ahler manifold which is not a surface must be totally geodesic, cf. \cite{sven} (in particular, it must be a homothetic submersion). In this case it will be a stable $\sigma_2$--critical map.

\begin{re}
The existence of nontrivial homothetic $(\phi,J)$-holomorphic submersions from a Sasakian to a K\"ahler manifold is less restricted, as shown by the example of the Hopf map. 

In order to analyse their stability we can easily rewrite the $\sigma_2$-Hessian for $\sigma_2$-critical $(\phi,J)$-holomorphic submersions from a Sasakian manifold $(M^{2m+1}, \phi, \xi, \eta, g)$ to a K\"ahler one, starting from the following analogue of the relation \eqref{dfi}, true for all $k=1,...,m$:
\begin{equation}\label{dfii}
\begin{split}
&h\left((\nabla^{\varphi})_{e_k, e_k}^{2}v + (\nabla^{\varphi})_{\phi e_k, \phi e_k}^{2}v + R^N(v, \dif \varphi(e_k))\dif \varphi(e_k)
+ R^N(v, \dif \varphi(\phi e_k))\dif \varphi(\phi e_k) , \ w \right) \\
&=-h\left(\Dd^{\varphi}v(e_k), \ \Dd^{\varphi}w(e_k)\right)- [g(\nabla_{e_k}X_0, e_k) + g(\nabla_{\phi e_k}X_0, \phi e_k)] +
2h(\nabla^{\varphi}_{\xi}v, J^N w),
\end{split}
\end{equation}
where $\Dd^{\varphi} v (X):=\nabla^{\varphi}_{\phi X} v - J^N \nabla^{\varphi}_X v$ and $X_0$ is defined similarly.
\end{re}

\section{Applications}
\subsection{A revisited stability result for full $\sigma_{1,2}$-energy}
In this subsection we deal with stability of homothetic (local) diffeomorphisms for the full energy $\mathcal{E}_{\sigma_{1,2}}$ defined by the relation \eqref{generg}. Essentially we prove again in a general context a stability result from \cite{los, man}, our general approach showing clearly that it is specific to 2 and 3-dimensional cases.

Obviously, a mapping $\varphi$ will be $\sigma_{1,2}$--critical if it satisfies the following Euler-Lagrange equations
\begin{equation*}
\tau(\varphi) + \kappa  \tau_{\sigma_2}(\varphi)=0.
\end{equation*}
Moreover, it will be stable if
\begin{equation}\label{fullhess}
\begin{split}
\mathrm{Hess}^{\sigma_{1,2}}_{\varphi}(v,v) =& 
\mathrm{Hess}^{\sigma_{1}}_{\varphi}(v,v) + \kappa \mathrm{Hess}^{\sigma_{2}}_{\varphi}(v,v)
\end{split}
\end{equation}
is positive for all $v$.

Let us suppose that $\varphi$ is a HH map between spaces of equal dimensions $m=n$ (if $n\geq 3$ it is a homothetic local diffeomorphism, cf. \cite[Theorem 11.4.6]{ud}), which is always $\sigma_{1,2}$-critical.
As in the proof of Proposition \ref{sig2stable}, replacing $v=\lambda^{-2}\dif \varphi (X_v)$, we get the following expression for $\mathrm{Hess}^{\sigma_{1}}_{\varphi}(v,v)$:
\begin{equation*}
\begin{split}
\int_M  \left[ \abs{\nabla^{\varphi} v}^2 - h(\mathrm{Ric}^{\varphi}v, v) \right] \nu _g &=
\int_M  \lambda ^{-2}\left[\abs{\nabla X_v}^2 - \mathrm{Ric}^{M}(X_v, X_v)\right]\nu _g\\
&= \int_M \lambda^{-2}\left\{ \frac{1}{2}\abs{\mathcal{L}_{X_v}g}^2 -(\mathrm{div}X_v)^2  \right\}\nu _g. 
\end{split}
\end{equation*}

Combining with \eqref{hsig2}, we obtain the full Hessian in the form:
\begin{equation*}
\begin{split}
\mathrm{Hess}^{\sigma_{1,2}}_{\varphi}(v,v)
=& \int_{M} \left\{\frac{\lambda ^{-2}+ (n-2)\kappa}{2} \abs{\mathcal{L}_{X_v}g}^2 -(\lambda ^{-2}+ (n-3)\kappa)(\mathrm{div}X_v)^2 \right\} \nu _g
\end{split}
\end{equation*}

Using again the inequality \eqref{newton}, we have 
\begin{equation*}
\mathrm{Hess}^{\sigma_{1,2}}_{\varphi}(v,v) \geq \int_M \left[ \frac{2}{n}(\lambda ^{-2}+ (n-2)\kappa) - (\lambda ^{-2}+ (n-3)\kappa) \right]
(\di X_v)^2 \nu_g .
\end{equation*}

The right hand side term can be positive (for non-constant maps) only when $n=2$ (trivially) and $n=3$. We can conclude as follows:

\begin{pr}\label{hld}
A homothetic local diffeomorphism between $3$-dimensional manifolds is a stable critical point for $\mathcal{E}_{\sigma_{1,2}}$ if
\begin{equation}\label{manton}
\lambda \geq \frac{1}{\sqrt{2 \kappa}}.
\end{equation}

Supposing that on $M$ there exists a conformal vector field which is not Killing, then no homothetic local diffeomorphism defined on $M$ can be a stable critical point for $\mathcal{E}_{\sigma_{1,2}}$ if one of the following conditions holds:

\noindent $(i)$ \  $n=3$ and $\lambda < \frac{1}{\sqrt{2 \kappa}}$; \qquad
\noindent $(ii)$ \ $n \geq 4$.
\end{pr}

\begin{re}
\ ($a$) The condition \eqref{manton} coincides with the one found in \cite{los, man} for the identity map $\mathrm{Id}_{M^3}$ (taking $\kappa=1$). 

\medskip
\noindent ($b$) According to \eqref{identhh}, if a homothetic local diffeomorphism (in particular, $\mathrm{Id}_{M^m}$, $m \geq 2$) is stable (as critical point) for the Dirichlet energy, then it is stable also for the full energy $\mathcal{E}_{\sigma_{1,2}}$, for all values of the coupling constant and the dilation factor. This will be true for any compact Riemannian manifold of constant curvature except for the standard unit sphere, cf. \cite{urak}. 

\medskip
\noindent ($c$) When $m=n=3$ and $\kappa=1$, it has been proved \cite{los} that if $\lambda \geq 1$, then diffeomorphic homotheties are, up to isometries, the only absolute minimizers of the Skyrme energy among all maps of a given degree. Notice however that for mappings between 3-spheres of different radii (with canonical metrics) this result provides little information as a homothety must have degree 1 in this case.
\end{re}

Another example of stable critical map for the \textit{full} energy $\mathcal{E}_{\sigma_{1,2}}$ is the Hopf map, according to \cite[Theorem 5.3]{sve}. Analogous to Proposition \ref{hld}, we have stability only when the coupling constant exceeds a critical value ($\kappa \geq 1$). Nevertheless the nature of the argument is different.

\subsection{Constructing relevant maps between spheres} In this final part we shall focus on examples of interest for the original physical models, that is mainly mappings between spheres: $\Ss_{R}^3 \to \Ss^3$ and $\Ss_{R}^3 \to \Ss^2$.

Let us first recall how can we parametrise the unit 3-sphere:

\begin{enumerate}
\item[(a)] (Join of circles: $\Ss^{3} = \Ss^{1} \ast \Ss^{1}$).  
$$
\{\left(\cos s \cdot e^{\ii x_1}, \ \sin s \cdot e^{\ii x_2} \right) \ \vert \ (x_1, x_2, s) \in [0, 2\pi)^2 \times [0, \pi / 2] \}.
$$ 
Then the standard Riemannian metric and contact form are given by: 
$$
g = \cos^{2}s \, \dif x_1^2 + \sin^{2}s \, \dif x_2^2 + \dif s^2 ; \quad 
\eta = \cos^{2}s \, \dif x_1 - \sin^{2}s \, \dif x_2.
$$

\item[(b)] (Suspension of the 2-sphere: $\Ss^{3} = \Ss^{0} \ast \Ss^{2}$).  
$$
\{\left(\cos s , \ \sin s \cdot (\cos t, \ \sin t \, e^{\ii x})\right) \  \vert \  (x, s, t) \in [0, 2\pi) \times [0, \pi]^2 \}.
$$ 
Then the standard Riemannian metric and contact form are given by: 
$$
g = \dif s^2 + \sin^{2}s \left( \dif t^2 + \sin^{2}t \, \dif x^2 \right),
\quad
\eta = \cos t \dif s - \frac{\sin 2s}{2}\sin t \dif t
+ \sin^2 s \sin^2 t \dif x.
$$

\item[(c)] (Unit tangent bundle of $\Ss^{2}$).  
$$
\{e^{\ii \theta}\left(\cos s \, \textbf{x} + \ii \sin s \, \textbf{y}\right) \  \vert \  (\theta, s) \in [0, 2\pi) \times [0, \pi /4], (\textbf{x}, \textbf{y}) \in S_{2,2} \},
$$ 
where $S_{2,2}=\{(\textbf{x},\textbf{y}) \in \RR^2 \oplus \RR^2  \ \vert \ \norm{\textbf{x}}=\norm{\textbf{y}}=1, \langle \textbf{x} , \textbf{y} \rangle =0\}$ is the Stiefel manifold of orthonormal 2-frames in the plane. Taking into account that $(\theta, (\textbf{x},\textbf{y}))$ and $(\theta + \pi, (-\textbf{x}, -\textbf{y}))$ corresponds to the same point we get a parameterisation of $\Ss^{3}$ (for more details about this construction, see \cite{ud, xi}).\\ 
The standard Riemannian metric and contact form are given by: 
$$
g = \dif s^2 + \dif \theta^2 + \dif \mu^2 + 2\sin(2s) \cdot \dif \theta \dif \mu ,\quad 
\eta = \dif \theta + \sin 2s \, \dif \mu
$$
where we have replaced $\textbf{x} = (\cos \mu, \sin \mu)$,  $\textbf{y}=(\sin \mu, -\cos \mu)$, $\mu \in [0, 2\pi)$.
\end{enumerate}
\begin{re}
Adapting the above parameterisations to a radius $R$ sphere is obvious. Nevertheless, the induced metric is no more an associated metric for the standard (induced) contact form. To have a contact metric structure on $\Ss_{R}^{3}$ one should either rescale both metric and contact form (with $R^{-2}$ and $R^{-1}$ respectively) or "squash" the sphere by taking the metric $\tilde g = R^{-1}g + (1 - R^{-1})\eta \otimes \eta$ (see \cite[pg 50]{ble} for more details).
\end{re}
On the other hand, the two-sphere $\Ss^2$ will be always parametrised as the suspension of the circle, that is 
$\{\left(\cos u , \ \sin u \cdot e^{\ii v})\right) \  \vert \  (u, v) \in  [0, \pi] \times [0, 2\pi) \}$ so that the standard metric and the symplectic (area) 2-form are given by:
$$
h = \dif u^2 + \sin^{2}u \, \dif v^2 ; \quad \Omega = -\frac{1}{2}\sin u \dif u \wedge \dif v. 
$$

Let us analyse now some concrete mappings onto spheres (endowed with the standard metric unless otherwise stated) that exhibit some symmetries, including the \textit{equivariance} with respect to (composition with) isoparametric functions cf. \cite[Chapter 13]{ud}.

\begin{ex}[Skyrme's Hedgehog map] Let $\varphi_{f}: \RR^{3} \to \Ss^3$ be given by
\begin{equation}\label{hedge}
r \left(\cos t , \ \sin t \cdot  e^{\ii x}\right) 
\mapsto 
\left(\cos f(r) , \ \sin f(r) (\cos t , \ \sin t \cdot  e^{\ii x})\right) ,
\end{equation}
where $f(0)=\pi$ and $f(\infty)=0$. This is a degree 1 map whose eigenvalues and corresponding Skyrme equations can be found in \cite{mant}.
\end{ex}

\begin{ex}[Suspension of mappings between 2-spheres] Let $\varphi_{f,q,r}: \Ss^{3} \to \Ss^3$ be given by
\begin{equation}\label{fmap}
\left(\cos s , \ \sin s (\cos t , \ \sin t \cdot  e^{\ii x})\right) 
\mapsto 
\left(\cos f(s) , \ \sin f(s) (\cos q(t,x) , \ \sin q(t,x) \cdot  e^{\ii r(t,x)})\right) ,
\end{equation}
where $f$, $q$, $r$ are smooth real functions obeying appropriate boundary conditions. When they are suspensions of holomorphic maps between 2-spheres, these maps constitute the \textit{rational map ansatz} \cite{rat} (after arranging the domain of definition to be $\RR^3$ as in the previous example). Their eigenvalues $\lambda_{1}^{2}$, $\lambda_{2}^{2} = \lambda_{3}^{2}$, one of them with double multiplicity due to the transversal holomorphicity of the map, and the Euler-Lagrange equations for full Skyrme energy are explicitly given in \cite{rat}. No solution in close form is known, but in most cases its numerical approximation provides the lowest energy field configurations available until now.
\end{ex}

\begin{ex}[$\alpha$-joins] Let $\psi^{\alpha}_{k, \ell}: \Ss^{3} \to \Ss^3$ be given by
\begin{equation}\label{aho}
(\cos s \cdot e^{\ii x_1}, \ \sin s \cdot e^{\ii x_2}) \mapsto 
\left( \cos \alpha (s) e^{\ii k x_1}, \  \sin \alpha (s) \cdot e^{\ii \ell x_2} \right),
\end{equation}
where $k, \ell \in \ZZ^*$ and $\alpha:[0, \pi /2] \to [0, \pi /2]$ satisfies the boundary conditions $\alpha(0)=0$, $\alpha(\pi /2)=\pi /2$. These maps are \textit{equivariant} (with respect to projections to the $s$-parameter) and have the degree $k\ell$. Their Cauchy-Green tensor's eigenvalues are
$$ \lambda_{1}^{2}=[\alpha^{\prime}(s)]^2, \quad
\lambda_{2}^{2}=\ell^2\frac{\sin^2 \alpha}{\sin^2 s}, \quad
\lambda_{3}^{2}=k^2\frac{\cos^2 \alpha}{\cos^2 s}.$$
They are critical points for full Skyrme energy if $\alpha$ is a solution of an ODE, explicitly given in \cite{JMW}. Note that, without refering to the principle of symmetric criticality, one can directly check that two of the equations in \eqref{sig3} are trivially satisfied and the third one gives us precisely this ODE in $\alpha$ (by adding the harmonicity equation).\\
Recall that in \cite{JMW} it was conjectured that the true minima of degree (baryon number) $B=2$ is of this form. Imposing the condition of contactomorphic type $\lambda_{3}^{2} = \lambda_{1}^{2}\cdot\lambda_{2}^{2}$ and $k=2$, $\ell=1$, we get the profile function $\alpha(s)=\arccos(\cos^2 s)$ (be aware that we do not obtain a contactomorphism). The corresponding field configuration $\psi^{\alpha}_{k, \ell}$ on $\Ss^{3}_{R}$ has a total energy ratio $\mathcal{E}_{\texttt{Skyrme}}/12\pi^2$ of $1.05175$ (after minimization with respect to the radius $R$), close to the value $1.047762$ found in \cite{JMW}. Compare with $\mathcal{E}_{\texttt{Skyrme}}/12\pi^2=1$ for the identity map.\\
A direct computation shows that this particular configuration can be rendered $\sigma_{2}$ or $\sigma_{1,2}$-critical by an appropriate conformal change of metric, using Theorem \ref{bic}.  
\end{ex}

\begin{ex}[Nomizu-equivariant maps, \cite{xi}] Let $\zeta^{\alpha}_{k}: \Ss^{3} \to \Ss^{3}$ be given by:
\begin{equation}\label{aho}
e^{\ii \theta}\left(\cos s \, \textbf{x} + \ii \sin s \, \textbf{y}\right) \mapsto 
e^{\ii k\theta}\left(\cos \alpha(s) \, \textbf{x} + \ii \sin \alpha(s) \, \textbf{y}\right),
\end{equation}
where $k$ is an odd integer and $\alpha:[0, \pi /4] \to [0, \pi /4]$ satisfies the boundary conditions $\alpha(0)=0$, $\alpha(\pi /4)=\pi /4$. These maps are \textit{equivariant} (with respect to projections to the $s$-parameter) and of degree $k$. 

The system \eqref{sig3} reduce to one single equation:
\begin{equation*}
\alpha^{\prime \prime} (k^2 -2k \sin 2\alpha \sin 2s +1) -
2k(\alpha^{\prime})^2 \sin 2s \cos 2\alpha +
2\alpha^{\prime} \left[(k^2 + 1)\tan 2s - 2k \frac{\sin 2\alpha}{\cos 2s}\right] + k^2 \sin 4\alpha = 0.
\end{equation*}
Note that $\alpha(s)=s$ is a solution for $k=1$ (i.e. $\mathrm{Id}_{\Ss^3}$ is $\sigma_2$-critical). It would be interesting to find higher degree solutions in close form.
\end{ex}

\begin{ex}[Absolute minima for the strongly coupled Faddeev model]
Let $\varphi_{\gamma, k}: \Ss^3 \to \Ss^2$ be defined by, cf. \cite[Ex.2]{slub}:
\begin{equation}
e^{\ii \theta}\left(\cos s \, \textbf{x} + \ii \sin s \, \textbf{y}\right) \mapsto \left(\cos \gamma(s) \ , \ \sin \gamma(s) \cdot e^{\ii k\mu}\right),
\end{equation}
where $\gamma:[0, \pi / 4] \to [0, \pi]$ and $k$ is an even integer. For $\gamma(s)=\frac{\pi}{2}\pm 2s$ and $k=2$ we obtain again the Hopf map.
The eigenvalues of this submersion are
$$
\lambda_{1}^{2}=0; \quad 
\lambda_{2}^{2} = \left[\gamma^{\prime}(s)\right]^2 ; \quad
\lambda_{3}^{2}= \frac{k^2 \sin^2 \gamma(s)}{\cos^2 2s} ,
$$
with the corresponding orthonormal frame of eigenvectors
$$
E_{1}=\frac{\partial}{\partial \theta}; \quad 
E_{2}=\frac{\partial}{\partial s}; \quad
E_{3}=\tan 2s \frac{\partial}{\partial \theta} - \frac{1}{\cos 2s} \frac{\partial}{\partial \mu}.
$$
As the vertical vector $E_1$ is the Reeb vector field of the standard contact structure of $\Ss^3$, the fibers are minimal, so the $\sigma_2$-Euler-Lagrange equations \eqref{fh} reduce in this case to $\gr (\lambda_{2}^{2}\lambda_{3}^{2})=0$. This equation is obviously solved by the same $\gamma$ as for Hopf map, for  any (even) value of $k$. To calculate the Hopf invariant of this map, notice first that:
$$
\varphi_{\gamma, k}^{*}\Omega = \frac{k}{2}\dif \eta ,
$$
where $\eta$ is the standard contact 1-form given above. So, by definition
$$
Q(\varphi_{\gamma, k})=\frac{1}{4\pi^2}\int_{\Ss^3}\frac{k^2}{4}\dif\eta \wedge \eta =\frac{k^2}{8\pi^2}\mathrm{Vol}(\Ss^3)=\frac{k^2}{4}.
$$
Notice that $\mathcal{E}_{\texttt{Faddeev}}^{\infty}(\varphi_{\gamma, k})=\mathcal{E}_{\sigma_2}(\varphi_{\gamma, k})=4\pi^2 k^2$ and therefore the topological lower bound in \cite{svee}, $$\mathcal{E}_{\sigma_2}(\varphi) \geq 16\pi^2 Q(\varphi)$$ is attained, i.e. $\varphi_{\gamma, k}$ is a global minima in its homotopy class for the strongly coupled Faddeev model on $\Ss^3$.

\end{ex}
\begin{ex}
Consider the $\alpha$-\textit{Hopf construction} \cite{ratto}:
\begin{equation}\label{aho}
\varphi: \Ss^{3} \to \Ss^2, \qquad 
(\cos s \cdot e^{\ii x_1}, \ \sin s \cdot e^{\ii x_2}) \mapsto 
\left( \cos \alpha (s), \  \sin \alpha (s) \cdot e^{\ii (kx_1+\ell x_2)} \right),
\end{equation}
where $k, \ell \in \ZZ^*$ gives the Hopf invariant $Q(\varphi)=k\ell$ and $\alpha:[0, \pi /2] \to [0, \pi]$ satisfies the boundary conditions $\alpha(0)=0$, $\alpha(\pi /2)=\pi$. When $(k, \ell)=(\mp 1, 1)$ and $\alpha(s) = 2s$, this construction provides the (conjugate) Hopf fibration.\\
Now take $\alpha(s) = 2s$ so that $\varphi$ corresponds (via the composition with a version of stereographic projection) to a higher charged configuration proposed in \cite{war}. Notice that: $\varphi^* \Omega = \dif \eta_{k, \ell}$, where $\eta_{k, \ell} =k\cos^{2}s \, \dif x_1 - \ell \sin^{2}s \, \dif x_2$ defines a contact form on $\Ss^{3}$ with volume element
$\frac{1}{2}\eta_{k, \ell} \wedge \dif \eta_{k, \ell}= k\ell \sin s \cos s \dif s \wedge \dif x_1 \wedge \dif x_2$.
Inspired by the general Boothby-Wang construction described in Example \ref{BWfibr}, let us endow the 3-sphere with the metric
$$
g_{k, \ell} = \frac{1}{4}\varphi^*h + \eta_{k, \ell} \otimes \eta_{k, \ell}
= k^2 \cos^{2}s \, \dif x_1^2 + \ell^2 \sin^{2}s \, \dif x_2^2 + \dif s^2. 
$$
Then we can check that the fibers of $\varphi$ are minimal with respect to $g_{k, \ell}$, by verifying the identity $\mathcal{L}_{\xi}\eta_{k, \ell}=0$, where $\xi=k^{-1}\frac{\partial}{\partial x_1}-\ell^{-1}\frac{\partial}{\partial x_2}$ is an unitary vertical vector field (and the Reeb field of the contact structure).
As by construction our map is horizontally homothetic with dilation $\lambda=2$, applying Corollary \ref{crit+}$(ii)$, we can conclude that 
$\varphi : (\Ss^3, g_{k, \ell}) \to (\Ss^2, h_{\texttt{standard}})$ is a critical map for \textit{full} Faddeev energy on the "squashed" sphere. Its Hopf invariant is $k\ell$, its energy is proportional to the standard energy of the Hopf map, i.e. $\mathcal{E}_{\sigma_{1,2}}^{g_{k, \ell}}(\varphi)= Q(\varphi) \cdot \mathcal{E}_{\sigma_{1,2}}^{g_{1, 1}}(\varphi)$, and when $k=\ell=1$ it coincides with the standard Hopf map on the round 3-sphere. 
\end{ex} 

As we have seen in Theorem \ref{contactomo}, contactomorphisms (i.e. mappings that preserves contact structures, $\varphi^* \eta ' = a\eta$) are close to both harmonic and $\sigma_2$--maps. When they are moreover transversally holomorphic (or commutes with the almost contact structures), their Cauchy-Green spectrum is $\{a^2, a, a\}$, that is they are contact homotheties. But the existence of a contact homothety imposes severe restrictions on the curvatures of the domain and codomain. Indeed, if it exists a contact homothety between Sasakian space forms $M$ and $N$, then their $\phi$-sectional curvatures must obey the relation:
$$
\frac{c_M +3}{a} - 3 = c_N.
$$
But \textit{a priori} the possibility to find (general) contactomorphisms that solve Skyrme equations remains open. Moreover, an estimation of their energy for small degrees shows that it could be lower than in the rational map ansatz case. We end this paragraph with a short list of contactomorphisms, emphasizing those that depend on a free function, to be determined in order to obtain $\sigma_2$-critical maps.

\begin{ex}[Contactomorphisms]\label{tact}
$(i)$ (From the Heisenberg group.) Consider $\RR^3$ endowed with its standard contact structure $\eta = \dif z - y\dif x$ and the associated (Sasakian) metric $g = \dif x^2 +
\dif y^2 + \eta \otimes \eta$ of $\phi$-sectional curvature $-3$.
Then $\delta_a: \RR^3 \to \RR^3$ given by $\delta_a (x, y, z) := (ax, ay, a^2 z)$, $a \in \RR^+$, is a contact homothety and therefore $\sigma_{1,2}$-critical.\\
In the same context, "shift" contactomorphisms $\varsigma_f (x, y, z) := (x, ay + f^{\prime}(x), a z + f(x))$, $a \in \RR^+$ depend on a free function (if \eqref{contactsig3} is compatible in $f$, we get a $\sigma_2$-critical map).\\
For a contactomorphism with $\lambda_{1}^{2}=\lambda_{2}^{2}\lambda_{3}^{2}=1$, $\lambda_{2,3}^{2}$ nonconstant, between $(\RR^3, \eta, g)$ and $(\RR^3, \eta^{\prime} = \cos z \dif x + \sin z \dif y, g^{\prime}=\dif x^2 + \dif y^2 +\dif z^2 )$, see \cite[pg.100]{ble}.\\
For a contactomorphism between $(\RR^3, \eta, g)$ and $\Ss^3 \setminus \{p\}$ with standard contact structure and metric, see \cite[Prop. 2.1.8]{gei}.

\medskip
$(ii)$ (Between 3-tori.) Consider the torus $\TT^3$ endowed with its standard contact structure $\eta = \cos z \dif x + \sin z \dif y$ and the associated (non-Sasakian) flat metric 
$g = \dif x^2 + \dif y^2 + \dif z^2$.
Let $\varphi_{a}^{f}: \TT^3 \to \TT^3$ be given by 
$$(x, y, z) \mapsto (ax - \int_{0}^{z} \tan s \cdot f'(s)\dif s, ay + f(z), z),$$
where $a \in \ZZ$ and $f$ is an arbitrary periodic function. Then $\varphi_{a}^{f}$ is a contactomorphism, its eigenvalues satisfying $\lambda_{1}^{2}=\lambda_{2}^{2}\lambda_{3}^{2}=a^2$. As $g$ is not a $K$-contact metric, in order to find $f$ such that $\sigma_{1,2}$-critical, we must use \eqref{contactsig3}. 

\medskip
$(iii)$ (Between 3-spheres.) Let $\zeta^{\alpha, \beta}_{k}: \Ss^{3} \to \Ss^{3}$ be given by:
\begin{equation*}
e^{\ii \theta}\left(\cos s \, \textbf{x}(\mu) + \ii \sin s \, \textbf{y}(\mu)\right) \mapsto 
e^{\ii k\theta}\left(\cos \alpha(s, \mu) \, \textbf{x}(\beta(\mu)) + \ii \sin \alpha(s, \mu) \, \textbf{y}(\beta(\mu))\right),
\end{equation*}
where $k \in \ZZ$ and $\alpha, \beta$ satisfy appropriate boundary conditions. One can directly check that if we take 
$\alpha(s, \mu)=\frac{1}{2}\arcsin\left(k A(\mu)\sin 2s \right)$ and $\beta(\mu)=\int_{0}^{\mu}A(t)^{-1}\dif t$, $A$ being an arbitrary function, then $\zeta^{\alpha, \beta}_{k}$ is a contactomorphism (i.e. $\varphi^* \eta = k\eta$). 

\end{ex}

\end{document}